\documentclass[11pt]{amsart}
\usepackage{amsthm}
\usepackage{enumitem}
\usepackage{mathtools}
\usepackage{moreenum}
\usepackage{amssymb,verbatim}
\usepackage{amsmath}
\usepackage{amssymb}
\usepackage{amsfonts}
\usepackage{amscd}
\usepackage{bbm}
\usepackage{mathrsfs}
\usepackage{upgreek}
\usepackage{xspace,cmap}
\usepackage{csquotes}
\usepackage{stmaryrd} 
\usepackage[colorlinks,citecolor=blue,urlcolor=black,linkcolor=black]{hyperref}
\newcommand{\norm}[1]{\left\lVert#1\right\rVert}
\usepackage{mathtools}

\usepackage{pst-node}
\usepackage{tikz-cd} 

\newcommand{\N}{\omega}

\newcommand{\cS}{\mathcal{S}}
\newcommand{\nat}{\omega}

\renewcommand{\restriction}{\mathbin\upharpoonright}    
\newtheorem*{theorem*}{Theorem}
\newtheorem*{maintheorem*}{Main Theorem}
\newtheorem*{corollary*}{Corollary}

\newtheorem*{definition*}{Definition}

\newtheorem{mthm}{Theorem}
\newtheorem{theorem}{Theorem}[section]

\newtheorem{prop}[theorem]{Proposition}
\newtheorem{claim}{Claim}[theorem]

\newtheorem{lemma}[theorem]{Lemma}
\newtheorem{cor}[theorem]{Corollary}

\newtheorem{question}{Question}
\newtheorem{fact}[theorem]{Fact}

\theoremstyle{definition}

\newtheorem{definition}[theorem]{Definition}
\newtheorem{notation}[theorem]{Notation}

\newtheorem{Setup}[theorem]{Setup}

\theoremstyle{remark}
\newtheorem{remark}[theorem]{Remark}

\newcommand*\axiomfont[1]{\textsf{\textup{#1}}}

\newcommand\gch{\axiomfont{GCH}}

\newcommand\ch{\axiomfont{CH}}

\newcommand\ale[1]{\marginpar{Alejandro: #1}}

\DeclareRobustCommand{\rchi}{{\mathpalette\irchi\relax}}
\newcommand{\irchi}[2]{\raisebox{\depth}{$#1\chi$}}

\DeclareMathOperator{\Suc}{Succ}

\DeclareMathOperator{\con}{conv}

\DeclareMathOperator{\supp}{supp}

\DeclareMathOperator{\range}{Im}

\DeclareMathOperator{\Part}{Part}
\DeclareMathOperator{\Null}{Null}

    \def\sq{\sqsubseteq}
    
    \newcommand{\one}{\mathop{1\hskip-3pt {\rm l}}}

\makeatletter
\newcommand{\tpitchfork}{%
  \vbox{
    \baselineskip\z@skip
    \lineskip-.52ex
    \lineskiplimit\maxdimen
    \m@th
    \ialign{##\crcr\hidewidth\smash{$-$}\hidewidth\crcr$\pitchfork$\crcr}
  }%
}
\makeatother

\def\s{\subseteq}
\def\forces{\Vdash}

\DeclareMathOperator{\cf}{cf}
\DeclareMathOperator{\cov}{cov}

\DeclareMathOperator{\non}{non}

\renewcommand{\mid}{\mathrel{|}\allowbreak}

\newcommand{\dom}{\mathop{\mathrm{dom}}\nolimits}

\title[$L$-orthogonal sequences and $L$-orthogonal elements]{A Banach space with $L$-orthogonal sequences but without $L$-orthogonal elements}
\author[Avilés]{Antonio Avilés}
\address[Avil\'es]{Universidad de Murcia, Departamento de Matematicas, Campus de Espinardo, 
30100 Murcia, Spain}
\email{avileslo@um.es}
\author[Martínez-Cervantes]{Gonzalo Martínez-Cervantes}
\address[Martínez-Cervantes]{Universidad de Murcia, Departamento de Matematicas, Campus de Espinardo, 
30100 Murcia, Spain}
\email{gonzalo.martinez2@um.es}
\author[Poveda]{Alejandro Poveda}
\address[Poveda]{Harvard University, Department of Mathematics and Center of Mathematical Sciences and Applications, Cambridge, MA 02138, USA}
\email{alejandro@cmsa.fas.harvard.edu}
\urladdr{https://scholar.harvard.edu/apoveda/home}
\urladdr{www.alejandropovedaruzafa.com}
\author[Sáenz]{Luis Sáenz}
\address[Sáenz]{Centro de Ciencias Matemáticas, UNAM A.P. 61-3, Xangari, Morelia, Michoacán, 58089, México}
\email{luisdavidr@ciencias.unam.mx}



\begin{document}

\begin{abstract}
     We prove that the existence of Banach spaces with $L$-orthogonal sequences but without $L$-orthogonal elements is independent of the standard foundation of Mathematics, ZFC. This provides a definitive answer to \cite[Question~1.1]{AvilesMartinezRueda}. Generalizing classical $Q$-point ultrafilters, we introduce the notion of $Q$-measures  and provide several results generalizing former theorems by Miller \cite{Miller} and Bartoszynski \cite{Bartoszynski} for $Q$-point ultrafilters.
\end{abstract}

\maketitle
\section{Introduction}

The presence of isomorphic copies of \(\ell_1\) in a Banach space and their relation to orthogonality conditions in the corresponding bidual has been a central topic of research in Banach space theory. The theory was developed by G. Godefroy, N.~J. Kalton, and B. Maurey, among others top-notch experts, in the late 1980s. A fundamental result due to Godefroy  \cite[Theorem II.4]{god} states that a Banach space \(X\) contains an isomorphic copy of \(\ell_1\) if, and only if, \(X\) admits an equivalent renorming with an element \(x^{**} \in X^{**}\) such that for every \(x \in X\),
\begin{equation*}
\|x + x^{**}\| = 1 + \|x\|.
\end{equation*}
 These elements are termed \textit{L-orthogonal} as per the notation in \cite{loru}.

A similar characterization given by V.~Kadets, V.~Shepelska, and D.~Werner \cite[Theorem 4.3]{ksw} establishes that a Banach space \( X \) contains an isomorphic copy of \(\ell_1\) if, and only if, there exists an equivalent renorming of \( X \) with a sequence $(x_n)_{n\in \omega }$ in its unit ball such that for every \(x \in X\)
\begin{equation*}
\lim_{n\to \infty}\|x + x_n\| = 1 + \|x\|.
\end{equation*}

Sequences with this property are termed  \textit{$L$-orthogonal} in \cite{AvilesMartinezRueda}.
For an $L$-orthogonal sequence $(x_n)_{n\in \omega }$, the function $$\tau(x) := \lim_{n\to \infty} \|x + x_n\|$$ defines a type of Maurey satisfying \(\tau(x) = 1 + \|x\|\) for every \(x \in X\) \cite{Maurey}. Accordingly, \(L\)-orthogonal sequences are called \textit{canonical \(\ell_1\)-types} in \cite{ksw}.

Examples of nonseparable Banach spaces with \(L\)-orthogonal elements but with no $L$-orthogonal sequences are known. Whether there is  a Banach space with \(L\)-orthogonal sequences yet without \(L\)-orthogonal elements in its double dual was asked in \cite{AvilesMartinezRueda}. In the mentioned paper, the authors provide a  positive answer  under the existence of selective ultrafilters on the set of natural numbers, $\omega$. The existence of selective ultrafilters on $\omega$ is  consistent with the ZFC axioms --for instance, selective ultrafilters exist under the Continuum Hypothesis-- and so is its negation \cite{KunenTheorem}.

The above-mentioned result was later improved in \cite{HrusakSaenz} by the last author and Hru\v{s}\'ak, who showed that a free ultrafilter $\mathcal{U}$  is a $Q$-point if, and only if, for every Banach space $X$ and every $L$-orthogonal sequence $(x_n)_{n\in \omega}$ in $X$,  $\mathcal{U}$-$\lim x_n \in X^{**}$ is an $L$-orthogonal element. ($\mathcal{U}$-$\lim x_n$ is the standard notation for the limit, under the ultrafilter $\mathcal{U}$, in the weak*-topology of the sequence $(x_n)_{n\in \omega}$.)  In effect, this refines the result of \cite{AvilesMartinezRueda} in that every selective ultrafilter is a $Q$-point. 

\smallskip

In \cite{AvilesMartinezRueda}, the authors left open the question of whether, consistently, there exists a Banach space with $L$-orthogonal sequences but without $L$-orthogonal elements. This manuscript settles this question affirmatively:

\begin{mthm}[Independence of $L$-elements]\label{Main Main theorem}
The existence of a Banach space with an $L$-orthogonal sequence but without $L$-orthogonal elements in its bidual is independent of the $\mathrm{ZFC}$ axioms.  
\end{mthm}

	The key idea behind this result is the generalization of the notion of $Q$-point to measures, something already considered in \cite{Swierczynska}, and for $P$-measures in \cite{BorodulinSobota,BorodulinPpoints}. Remember that $\mathcal{U}$ is a $Q$-point if for every partition of $\omega$ into finite sets there is a set in $\mathcal{U}$ that is a selector for the partition, that is, a set that intersects each set of the partition in at most one point. There are several ways to extend this definition to measures, the most relevant for us are the following:
	
	\begin{definition}
		Let $\mu:\mathcal{P}(\omega)\longrightarrow [0,+\infty)$ be a finitely additive measure defined on the subsets of $\omega$ and vanishing on finite sets.
	\begin{enumerate}
		\item $\mu$ is a \textit{$Q^+(\omega)$-measure} if every partition of $\omega$ into finite sets has a selector of positive measure.
		\item $\mu$ is a \textit{ fit $Q$-measure} if for every partition of $\omega$ into finite sets and every $\delta<\frac{1}{2}$ there is a finite union of selectors of measure greater than $\delta\cdot \mu(\omega)$. 
		\item $\mu$ is a \textit{strong $Q$-measure} if every partition of $\omega$ into finite sets has a selector of full measure.
	\end{enumerate}
		
	\end{definition}

It is immediate that the Dirac measure of a $Q$-point is a strong $Q$-measure. Moreover, every strong $Q$-measure is a fit $Q$-measure and every fit $Q$-measure is a $Q^+(\omega)$-measure.

Theorem~\ref{Main Main theorem} follows from the combination of the following two facts.

\begin{theorem}\label{thm: no Q points1}
	It is consistent with ZFC that there are no $Q^+(\omega)$-measures.
\end{theorem}

\begin{theorem}\label{nofitnoL}
	If no fit $Q$-measures exist, then there is a Banach space with an $L$-orthogonal sequence and no $L$-orthogonal elements.
\end{theorem}

The first part of this paper (\S\ref{sec:Qmeasures}) is a purely set-theoretic investigation. In the first place in  \S\ref{sec: extending filters} we make sure that  these generalizations of $Q$-points are meaningful. Extending classical results in the area (see \cite[Theorem~4.6.6]{Bartoszynski}), we prove that $\mathfrak{d}=cov(\mathcal{M})$ is equivalent to the fact that every atomless measure of density less than $\mathfrak{d}$ extends to a strong $Q$-measure. So in some models these objects are indeed abundant. Then, in \S\ref{sec: no Q measures} we prove Theorem~\ref{thm: no Q points1}.  We generalize the proof that there are no $Q$-points in Laver's model \cite{Miller}, which applies to other models obtained by iterated forcing as well. A more general statement is that no $\sigma$-bounded-cc ideal is rapid in these models.

\smallskip

The second part of the paper (\S\ref{sec:BanachSpaceLorthogonal}) contains the proof of Theorem~\ref{nofitnoL}. This section focuses more on Banach spaces and less on set theory. We adapt the proof from \cite{AvilesMartinezRueda} that if no $Q$-points exist, then an $L$-orthogonal sequence can be constructed with no $L$-orthogonal elements in its weak$^*$-closure. There are two steps: First we prove that any measure $\mu:\mathcal P (\omega)\to [-1,1]$ whose positive and negative parts are not fit $Q$-measures yields  a Banach space $X_\mu$ with an $L$-orthogonal sequence $(x_n)_{n\in \omega}\s X_\mu$ such that  $T^{**}(\mu)$ is not an $L$-orthogonal element of $X_\mu$. (Here $T\colon \ell_1\rightarrow X$ is the operator defined as $T(e_n):=x_n$ for $n\in \omega$.) This is accomplished in \S\ref{sec: A space for each $Q$-measure}. Afterwards, in \S\ref{sec: A space without $L$-orthogonal elements}, we amalgamate all of these spaces $X_\mu$ into a single one that serves as a witness for Theorem~\ref{nofitnoL}. In addition to this proof, there is also in \S\ref{strongstarstar} a generalization of a result from \cite{HrusakSaenz} that goes in the converse direction: for an $L$-orthogonal sequence, $T^{**}(\mu)$ is always an $L$-orthogonal element if $\mu$ is a strong $Q$-measure.

\smallskip

The paper is concluded with Section \ref{sec:openquestions} by presenting a few open questions. Our notations are standard both in Set Theory and Banach Space Theory. Relevant preliminaries will be also provided next in \S\ref{sec:prelimminaries}.

\section{Preliminaries}\label{sec:prelimminaries}

As it is customary in Set Theory we denote by $\omega$ the set of natural numbers.
A set $\mathcal F\subseteq \mathcal P(\omega)$ is called a \emph{filter} if $\emptyset\notin\mathcal{F}$ and $\mathcal{F}$ is closed under taking supersets and finite intersections. We shall assume from the outset  that all of our filters contain the complements of finite subsets of $\omega$. 

{Ideals} provide the dual notion of a filter. Specifically, $\mathcal{I}\s\mathcal{P}(\omega)$ is an \emph{ideal} if, and only if, $\mathcal{I}^*=\{\omega\setminus I\mid  I\in \mathcal{I}\}$ is a filter; to wit,  $\omega\notin \mathcal{I}$ and $\mathcal{I}$ contains all finite subsets of $\omega$ and it is closed under finite unions and under taking subsets. 
The collection of \emph{$\mathcal{I}$-positive subsets of $\omega$} is  $$\mathcal{I}^+:=\{J\mid  J\notin \mathcal{I}\}.$$ The \emph{restriction} of an ideal (resp. a filter) to a set $X\s \omega$ is $$\mathcal{I}\lvert_X:=\{I\cap X:I\in \mathcal{I}\}.$$

\smallskip

Maximal filters (also known as \emph{ultrafilters})  are  of paramount importance. An ultrafilter  $\mathcal{U}\subseteq \mathscr{P}(\omega)$ is called a \emph{$Q$-point} if  for every partition $\{F_n\mid n\in \omega\}$ of $\omega$ into finite sets there exists $X\in \mathcal{U}$ such that $\lvert X \cap F_n\lvert\leq 1$ for all $n\in \omega$. More generally (and dually), we call an ideal $\mathcal{I}\subseteq \mathcal{P}(\omega)$  a \emph{$Q^{+}(\omega)$-ideal} if for every partition $\{F_n\mid n\in \omega\}$ of $\omega$ into finite sets there exists $X\in \mathcal{I}^+$ such that  $\lvert X \cap F_n\lvert\leq 1$ for all $n\in \omega$. A natural extension of $Q^{+}(\omega)$-ideals are  rapid ones.  An ideal \emph{$\mathcal{I}$ is rapid} if for every partition  $\{F_n\mid n<\omega\}$ of $\omega$ into finite sets there is $X\in \mathcal{I}^{+}$ such that, for each $n<\omega$, $|X\cap F_n |\leq n$. 

\smallskip


     

Observe that in all the definitions above, it is equivalent to say ``for every partition into finite sets'' and ``for every partition into finite intervals''. To see this, if $\mathcal{I}$ satisfies any of this notions for partitions into finite intervals and you have a general partition $\{F_n\}$ into finite sets, consider the partition $\{[m_n,m_{n+1})\}$ into intervals where $$m_{n+1} = \max\left(F_{2n+2}\cup \bigcup\{F_k : \min(F_k)<m_n\}\right).$$
 Notice that if $X\in\mathcal{I}
 ^+$ with $|X\cap [m_n,m_{n+1})|\leq 1$ for all $n$, then $|X\cap F_n|\leq 2$ for all $n$, but then $X$ can be divided into two pieces, one of which is in $\mathcal{I}^+$.

In what follows all Boolean algebras $\mathbb{B}$ are considered to be subalgebras of $(\mathcal{P}(\omega),\s, \cup,\cap,\setminus)$ and all of our measures on $\mathbb{B}$ will be assumed to be finitely-additive, of bounded variation and vanishing on all finite sets.
{For the record, recall that  $\mu$ is called \emph{atomless} if given any $A\in \mathbb{B}$ with $\mu(A)>0$ there is $B\s A$ such that $0<\mu(B)<\mu (A)$.}

\begin{definition}
    For $\varepsilon>0$, a measure $\mu\colon \mathbb{B}\to [0,1]$ is called an \emph{$\varepsilon$-strong $Q$-measure} if for any $\{A_n\mid n\in \omega\}$ partition of $\omega$ into finite sets there is $A\subseteq \omega$ such that $\mu(A)\ge\varepsilon$ and $\lvert A_n \cap A \lvert \leq 1$ for any $n\in \omega$.  Those measures that are $\mu(\omega)$-strong $Q$-measures will be referred just as strong $Q$-measures.
\end{definition}   
\begin{remark}
    Notice that any $Q$-point naturally defines a strong $Q$-measure. 
\end{remark}

Measures induce ideals in a natural way:
\begin{definition}
An ideal $\mathcal{I}\s\mathcal{P}(\omega)$ is called a \emph{measure ideal} (on $\mathbb{B}\subseteq \mathcal P (\omega)$) if there is a measure $\mu\colon \mathbb{B}\to [0,1]$ such that $\mathcal{I}=\Null(\mu)=\ker(\mu)=\{A\in\mathbb{B}\mid \mu(A)=0\}$. We say that
\emph{$\mu$ is a $Q^+(\omega)$-measure} if $\ker(\mu)$ is a $Q^+(\N)$-ideal. 
\end{definition}

The following notation will be used at several places in this paper:

\begin{notation}\label{notation:specialnotation}
    Consider $(x_n)_{n\in \omega }$ a sequence in a Banach space $X$, and a measure of bounded variation $\mu\colon \mathcal{P}(\omega)\to \mathbb{R}$. Let $T\colon \ell_1\to X$ be defined as $T(e_n):=x_n$ for  $n\in \omega$. Define the \emph{$\mu$-limit of $(x_n)_{n\in \omega }$} as $$\mu\text{-}\lim x_n:=T^{**}(\mu).$$ Here $T^{**}\colon \ell_1^{**} \rightarrow X^{**}$ and remember that $\ell_1^{**} = \ell_\infty^*$ is naturally identified with the set of finitely additive signed finite measures on $\mathcal{P}(\omega)$. In the case where $\mu$ stems from an ultrafilter $\mathcal{U}$ this coincides with the classical $\mathcal{U}$-limit.
\end{notation}

We consider $\mathcal{P}(\omega)$ equipped with the natural topology inherited from the product topology of $2^\omega=\{f\lvert f:\omega \to \{0,1\}\}$ via characteristic functions. At any point we say that a subset of $\mathcal{P}(\omega)$ has certain topological property (e.g., closed, Borel, or analytic) we shall be referring to this underlying topology.

\smallskip

In what follows we will consider several \emph{cardinal invariants}. 
A family $\mathcal{D}$ of functions in $\omega^\omega=\{f\lvert f:\omega \to \omega\}$ is called \emph{dominating} if for every $f\in  \omega^\omega$ there is $g\in \mathcal{D}$ such that $\{n\in \omega\mid g(n)\leq f(n)\}$ is finite. The cardinal $\mathfrak{d}$, known as the \textit{dominating number}, is defined as follows: 
$$\mathfrak{d}=\min\{\lvert \mathcal{D}\lvert: \mathcal{D}\s \omega^\omega \text{ is a dominating family}\}.$$

The dominating number $\mathfrak{d}$ admits several characterizations -- we will exploit a classical one, cf.~\cite{Blass2010}. Given  partitions $\mathcal{P}=\{P_n:n\in \N\}$ and $\mathcal{Q}=\{Q_n:n\in \N\}$ of $\N$ into finite intervals, we say that $\mathcal{P}$ \emph{dominates} $\mathcal{Q}$ (in symbols, $\mathcal{P}\geq \mathcal{Q}$) if for all but finitely many $n\in \omega$ there is $m\in \omega$ such that $Q_m\subseteq P_n$. Denote the set of interval partitions of $\N$ as $\Part$. A family $\mathcal D$ is \emph{dominant in $\Part$} if it is dominant in the order just introduced, meaning that for any $\mathcal{Q}\in \Part$ there is $\mathcal{P}\in\mathcal{D}$ such that $\mathcal{Q}\leq \mathcal{P}$. It is a classical result that the least cardinality of such a dominating family is $\mathfrak{d}$ \cite{Blass2010}.

\smallskip

Another central cardinal invariant is $\cov(\mathcal{M})$. A family $\mathcal{C}\subseteq 2^{\omega}$ is a covering family of meager sets if each set in $\mathcal{C}$ is meager and $\bigcup \mathcal{C}=2^{\omega}$. The covering number $\cov(\mathcal{M})$ is defined as the smallest cardinal of a covering family of meager sets; namely, $$\cov(\mathcal{M})=\min\{\lvert \mathcal{C}\lvert: \mathcal{C}\s 2^\omega \text{ is a covering family of meager sets}\}.$$

A \emph{reaping family} $\mathcal{R}\subseteq \mathcal{P}(\N)$ is a family such that for any $A\in \mathcal{P}(\N)$ there is $R\in\mathcal{R}$ such that $R\cap A$ or $A^c\cap R$ is finite. The least size of such a family is $\mathfrak{r}$, and it is known that $\cov(\mathcal{M})\leq \mathfrak{r}$, see \cite{Blass2010}.

\section{On  $Q^+(\omega)$-measures.}
\label{sec:Qmeasures}

\subsection{Models with many strong $Q$-measures.}\label{sec: extending filters}

A classical result in cardinal invariant theory asserts that any filter $\mathscr{F}$ of character less that $\mathfrak{d}$ can be extended to a $Q$-point if, and only if, $\cov(\mathcal{M})=\mathfrak{d}$ \cite[Theorem 4.6.6]{Bartoszynski}. Recall that if $\mathscr{A}\subseteq \mathscr{B}$ are families of sets, we say that $\mathscr{A}$ is dense in $\mathscr{B}$ if for any $B\in \mathscr{B}$ there is $A\in \mathscr{A}$ such that $A\subseteq B$. The least cardinality of a dense subfamily of a filter is called the character of that filter. For an algebra $\mathbb{B}$ of subsets of $\omega$, the density of $\mathbb{B}$ will be the least cardinality of a dense family of the infinite sets in $\mathbb{B}$.

\smallskip

The goal of this subsection is to prove the following result:

\begin{theorem}\label{thm: lots of Q measures} The following statements are equivalent:
\begin{enumerate}
    \item  $\mathfrak{d}=cov(\mathcal{M})$. 
    \item For any measure $\mu:\mathbb{B}\to [0,1]$, such that $[\omega]^{<\omega}\subseteq\mathbb{B}$, $\mu$ vanishes on finite sets, and has density less than $\mathfrak{d}$, there exists an atomless strong $Q$-measure $\nu:\mathcal P (\N)\to [0,1]$ extending $\mu$.
    \end{enumerate}
\end{theorem}


Before the proof, some notation and an observation. A \textit{partial selector} for $\mathcal{P} = \{P_n : n\in\omega\}$ is a set $S$ such that $|S\cap P_n|\leq 1$ for all but finitely many $n$. And we say that $S$ is $m$-skipping for $\mathcal{P}$ if for any $n\in \omega$, there is at most one $i\in [n,n+m]$ such that $\lvert P_i \cap S\lvert > 0$.

\begin{lemma}
    Let $\mathcal{P}=\{P_n:n\in \omega\}$ be a partition of $\omega$ in finite intervals, and $m\in \omega$. Then, there is another partition of $\omega$ in finite intervals $\mathcal{Q}$ with the property that if $S\subseteq \omega$ is a partial selector for $\mathcal{Q}$ then, $S$ is  a partial selector for $\mathcal{P}$ that is $m$-skipping.
\end{lemma}
\begin{proof}
    Just define $Q_n=\bigcup_{i\in [(m+1)n,(m+1)(n+1))}P_i$.
\end{proof}

For the next proof we need to introduce some notation. Given a Boolean algebra $\mathbb{B}\subseteq \mathcal{P}(\N)$ and $A_0, \dots, A_{n}\in \mathcal{P}(\N)$ we denote by $\mathbb{B}(A_0, \dots, A_n)$ the Boolean algebra generated by $\mathbb{B}\cup \{A_0, \dots, A_n\}$. We recall the fact that a measure $\mu$ on $\mathbb{B}$ can be extended to a measure $\tilde{\mu}$ on $\mathbb{B}(A)$ with $\tilde{\mu}(A)=r$ if and only if $\mu(C)\leq r \leq \mu(D)$ whenever $C,D\in\mathbb{B}$ and $C\leq A\leq D$, see \cite{Los}. Recall that $2^{<\omega}$ denotes the set of all finite sequences of $0$'s and $1$'s, while $2^{\omega}$ stands for all the infinite sequences with image in the same set. Given $t,s\in 2^{<\omega}$, $t^{ \frown}s$ denotes the concatenation of the sequences $s$ and $t$. 
Finally for $t\in 2^{<\omega}$,  the cone generated by $t$ is $$\langle t \rangle=\{f\in 2^\omega\mid \forall n \in \dom(t), t(n)=f(n)\}$$

\begin{proof}[Proof of theorem \ref{thm: lots of Q measures}]
   First assume $\mathfrak{d}=cov(\mathcal{M})$. Enumerate a dominating family $\mathscr{D}$ in $Part$ as $\mathscr{D}=\{\mathcal{P}_{\alpha}:\alpha < \mathfrak{d}\}$. We will construct a sequence of finitely additive measures that vanish on finite sets $$\{\mu_{\alpha}: \mathbb{B}_\alpha \to [0,1]: \alpha <\mathfrak{d}\},$$ such that

    \begin{enumerate}
    
    \item For any $\alpha \leq \beta<\mathfrak{d}$, $\mathbb{B}_\alpha\subseteq \mathbb{B}_\beta$,  $\mu_\beta\lvert_{\mathbb{B}_\alpha}=\mu_\alpha$, and 
    $\mathbb{B}_\alpha$ has density less than $\mathfrak{d}$. Also $\mu_0=\mu$.
\item For any $\alpha <\mathfrak{d}$ there is $D_\alpha$ a partial selector of $\mathcal{P}_\alpha$ that is 3-skipping and  $\mu_{\alpha+1}(D_\alpha)=1$.

\item For any successor ordinal $\alpha <\mathfrak{c}$, there is $T_\alpha:2^{<\omega}\to [\omega]^{\omega}$ such that:
\begin{enumerate}
    \item $T_\alpha(\varnothing)=\omega$.
    \item For any $s\in 2^{<\omega}$, $T_\alpha(s^{ \smallfrown}0)\cap T_\alpha(s^{ \smallfrown}1)=\varnothing$.
    \item For any $s\in 2^{<\omega}$, $T_\alpha(s^{ \smallfrown}0)\cup T_\alpha(s^{ \smallfrown}1)=T_\alpha(s)$.
    \item For any infinite $A\in \mathbb{B}_\alpha$, $s\in 2^{<\omega}$, $T_\alpha(s^{ \smallfrown}0)$ splits $A\cap T_\alpha(s)$.
\end{enumerate}
        
\end{enumerate}
    
The sequence is defined by recursion. Assume we have already defined $\mu_\alpha:\mathbb{B}_\alpha\to [0,1]$; our goal now is to define $\mu_{\alpha+1}$. Let $A_\alpha$ be dense in $\mathbb{B}_{\alpha}$, and consider $A'_\alpha$ all the infinite elements of $A_\alpha$.

Consider the partition $\mathcal{P}_\alpha$ and apply the previous lemma to $\mathcal{P}_\alpha$ and $m=4$ to get $\mathcal{Q}_{\alpha}=\{Q^{\alpha}_n:n\in \N\}$. Consider $$X=\Pi_{n\in \omega}Q^\alpha_n=\{f\in\omega ^\omega: \forall n \in \omega (f(n)\in  Q^\alpha_n)\}$$ and for any $A\in A_\alpha'$ let
$$H_{A}=\{f\in X: \exists^{\infty}n(f(n)\in A)\}.$$

Observe that this sets are all dense and $G_\delta$ in $X$, so, by the cardinal hypothesis there is $f\in \bigcap_{A\in A'_\alpha}H_A$; let $D_{\alpha}=\{f(n):n\in \omega\}$.
 
 We will construct $T_{\alpha+1}$ by recursion on the length of $s\in 2^{<\omega}$. Define $T_{\alpha+1}(\varnothing)=\omega$. If $T_{\alpha+1}(s)$ has already been defined, notice that $\{A\cap T_{\alpha+1}(s):A\in A_\alpha\cup \{D_\alpha\}\}$ cannot be reaping because $cov(\mathcal{M})\leq \mathfrak{r}$. So pick $T_{\alpha+1}(s^{ \smallfrown}0)\subseteq T_{\alpha+1}(s)$ that splits it. It is clear that $T_{\alpha+1}(2^{<\omega})\cap \mathbb{B}_{\alpha+1}=\{\omega\}$.

\bigskip
Let $\mathbb{B}_{\alpha+1}=\mathbb{B}_\alpha(T_{\alpha+1}(2^{<\omega})\cup \{D_\alpha\})$. We must extend $\mu_\alpha:\mathbb{B}
_\alpha\to [0,1]$ to  $\mu_{\alpha+1}:\mathbb{B}_{\alpha+1}\to [0,1]$. Notice that $D_\alpha$ cannot be contained in any element of measure less than 1, so we can extend the measure to $\mathbb{B}_\alpha(D_\alpha)$ assigning measure 1 to $D_\alpha$. In the next step, we observe that all elements of $\mathbb{B}_\alpha(D_\alpha)$ 
 that contain $T_{\alpha+1}(0)$ have measure 1, and those that are contained in $T_{\alpha+1}(0)$ have measure 0, so we can extend the measure to $\mathbb{B}_\alpha(D_\alpha, T_{\alpha+1}(0))$ assigning measure $1/2$ to $T_{\alpha+1}(0)$, and therefore also to $T_{\alpha+1}(1)$. In the next step, we observe that all elements of $\mathbb{B}_\alpha(D_\alpha,T_{\alpha+1}(0))$ 
 that contain $T_{\alpha+1}(00)$ have measure at least $1/2$, and those that are contained in $T_{\alpha+1}((0))$ have measure 0, so we can extend the measure to $\mathbb{B}_\alpha(D_\alpha, T_{\alpha+1}(0),T_{\alpha+1}(00))$ assigning measure $1/4$ to $T_{\alpha+1}(00)$, and thus also assigning $1/4$ to $T_{\alpha+1}(01)$. Continuing in this way, we can extend the measure to $\mathbb{B}_{\alpha+1}$ in such a way that  $\mu_{\alpha+1}(T_{\alpha+1}(s))=2^{-\lvert s \lvert}$ for all $s$.

Consider $\mu'=\bigcup_{\alpha<\mathfrak{d}}\mu_\alpha$. It is easy to see that this measure is atomless. If it does not have $\mathcal{P}(\omega)$ as its domain, extend it to an atomless measure in the obvious way. Now let us check that it is a strong $Q$-measure. It is enough to check that in the family $\{D_{\alpha}:\alpha <\mathfrak{d}\}$ there is an eventual partial selector for any partition of $\omega$ into finite sets. Eventual means that by removing a finite amount of elements it becomes a partial selector. So pick $\mathcal{P}=\{P_n:n\in \omega\}$ a partition of $\omega$ into finite sets. Define $(m_n)_{n\in \omega}$ recursively by $m_0=0$, $m_{n+1}=\max \bigcup \{P_k:m_n>\min P_k \}+1$.
    Let $Q_n=[m_n,m_{n+1})$ and $\mathcal{Q}=\{Q_n:n\in \omega\}$. Now find $\alpha<\mathfrak{d}$ such that  $ \mathcal{Q}\le \mathcal{P}_\alpha = \{P^{\alpha}_k:k\in \omega\}$.
    
    We claim that $D_{\alpha}$ is an eventual partial selector for $\mathcal{P}$. Notice that for any $n\in \omega$, the size of the set $\{m:Q_m\cap P_n\neq \varnothing\}$ is at most two, and in the case they are two they are successive. The same is true for any $m\in \omega$ and the set $\{k:Q_m\cap P^{\alpha}_k\neq \varnothing\}$. 
    So if $D$ is a partial selector of $\mathcal{P}_\alpha$ then for all but finite $n\in \omega$, $\lvert P_n \cap D\lvert \leq 4$.
    Notice that ${D}_{\alpha}$ is a selector of  $\mathcal{P}_\alpha$ that skips at least three indices until it selects another interval, so we may conclude it is a eventually partial selector of $\mathcal{P}$.  

    Assume now that we can extend any measure as mentioned in the statement. Consider $\gamma<\mathfrak{d}$, and $\{A_{\gamma} \subseteq 2^{\omega}:\alpha<\gamma\}$ a family of nowhere dense closed sets, we will show that $\bigcup_{\alpha<\gamma}A_{\gamma}\neq 2^{\omega}$. Define for any nowhere dense set $A\subseteq 2^{\omega}$, $$F_{A}=\{(n,s)\in \omega \times 2^{<\omega}:\forall t\in 2^{\le n}, \langle t^{ \frown}s\rangle\cap A=\varnothing \}.$$

    Notice that for any $\alpha,\beta <\gamma$,
    $F_{A_{\alpha}\cup A_{\beta}}\subseteq F_{A_\alpha}\cap F_{A_\beta}$. Also notice that if   $A$ is nowhere dense then for any $n\in \omega$ there is $s\in 2^ {<\omega}$ such that $(n,s)\in F_A$. These facts imply that $\{A_{F_\gamma}:\alpha< \gamma \}$ generates a filter. The second fact also implies we may define for any $\alpha<\gamma $, $f_{\alpha}\in \omega^\omega$ by

    $$f_{\alpha}(n)=\min\{k: \exists s\in 2^k ( (n,s)\in F_{A_\alpha})\}.$$

    The cardinality of the family implies we may find an increasing $f\in \omega^\omega$ such that for any $\alpha<\gamma $, $\{n: f_{\alpha}(n)\le f(n)\}$ is infinite. Let $F=\{(n,s):s\in 2^{\le f(n)}\}$, so  for any $\alpha<\gamma$, $F\cap{F_{A_\gamma}}$ is infinite. Let $m_0=0$ and $m_n=\sum_{i=0}^{m_{n-1}}f(i)$. Partition $F$ into $F_0$, $F_1$, with

    $$F_i=\bigg\{(n,s)\in F: n\in \bigcup_{j\in \omega}[m_{2j+i},m_{2j+i+1})\bigg\}$$

Let $l\in 2$ be such that $F_l$ has infinite intersection with every $F_{A_\alpha}$.
Consider the family $\{F_{A_{\alpha}} \cap F_l:\alpha<\gamma\}$. By hypothesis we may find $\mu:\mathcal{P}(F_l)\to [0,1]$ a strong $Q$-measure such that any element of this family has measure 1. Now consider for any $n\in \omega$, $$P_{n}=\{(j,s)\in F_l:j\in [m_{2n+l},m_{2n+l+1})\}.$$
Then $\{P_n:n\in \omega\}$ is a partition of $F_l$ into finite sets. Let $$X=\{(p_n,s_n):n\in \omega\}$$ be a selector of it such that $\mu(X)=1$. Let $x=s_0^{ \frown}s_1^{ \frown}\cdots ^{ \frown} s_n ^{ \frown}\dots \in 2^\omega$. We claim $x \notin \bigcup_{\alpha<\gamma}A_{\gamma}$. Fix $\alpha<\gamma$. Since $\mu(F_{A_{\alpha}}\cap F_l)=1$, there are infinite $n\in \omega$ such that $(p_n,s_n)\in F_{A_{\alpha}}\cap F_l\cap P_n$. Then, 
$$\lvert s_0^{ \frown}s_1^{ \frown}\cdots ^{ \frown} s_{n-1}\lvert \le \sum_{i=0}^{n-1}f(p_i)\le \sum_{i=0}^{m_{2(n-1)+l+1}}f(i)\le m_{2n+l}\le p_n.$$
This implies that $\langle s_0^{ \frown}s_1^{ \frown}\cdots ^{ \frown} s_n\rangle \cap A_\alpha=\varnothing$, and so  as $x\in \langle s_0^{ \frown}s_1^{ \frown}\cdots ^{ \frown} s_n\rangle$, we conclude that $x\notin \bigcup_{\alpha<\gamma}A_{\gamma}$.
\end{proof}

\subsection{Models without $Q^+(\omega)$-measures.}\label{sec: no Q measures}

The main goal of  this section is to prove the consistency of  the following statement: 
$$\text{ZFC + ``Every measure ideal on $\mathcal{P}(\omega)$ is not a $Q$-point''.}$$ 

 In what follows, we will prove a statement slightly stronger than this using forcing (see Theorem ~\ref{thm: no Q points}). Before presenting this consistency result we show that the Filter Dichotomy Axiom together with the non-existence of $Q$-points entail the non-existence of $Q^+(\omega)$-measures. 
 
 Given two filters $\mathscr{F}, \mathscr{G}$ over $\omega$ we say $\mathscr{F}$ is Rudin-Blass above $\mathscr{G}$ (in symbols $\mathscr{G}\leq_{RB}\mathscr{F}$) if there is $f:\omega\to \omega$ finite-to-one such that $\{X:f^{-1}(X)\in \mathscr{F}\}=\mathscr{G}$. The \emph{Filter Dichotomy Axiom} is the following assertion:
\begin{center}
    ``If $\mathscr{F}$ is any nonprincipal filter over $\omega$, either it is Rudin-Blass above the Fréchet filter or it is Rudin-Blass above an ultrafilter".
\end{center}

\begin{prop}
    Assume the filter dichotomy holds and there are no $Q$-points, then there are no $Q^+(\omega)$-measures.
\end{prop}

\begin{proof}
    Consider a measure $\mu:\mathcal{P}(\omega)\to [0,1]$ and assume there is $f:\omega\to \omega$ finite to one such that for any $X\subseteq \omega$, $X$ is cofinite if and only if $f^{-1}(X)=1$. Find a decreasing sequence of infinite sets $(A_n)_{n\in\omega}$ such that $\mu(f^{-1}(A_{n+1}))\leq \frac{1}{2}\mu(f^{-1}(A_{n}))$ and let $A$ be any pseudo-intersection of it. Then $A$ is infinite but $\mu(f^{-1}(A))=0$, a contradiction.

    By the filter dichotomy there must be $f\colon \omega\to \omega$ finite to one such that $\mathscr{U}=\{X:\mu(f^{-1}(X))=1\}$ is an ultrafilter. But we know there are no $Q$-points so there must be a partition of $\omega$ into finite sets $\mathcal{P}=\{P_n:n\in \omega\}$ such that no selector belongs to $\mathscr{U}$. Now define $\mathcal{Q}=\{f^{-1}(P_n):n\in \omega\}$ notice this is a partition of $\omega$ into finite sets. Pick any selector $S$ of $\mathcal{Q}$. Then $f(S)$ is a selector of $\mathcal{P}$, so $\omega\setminus f(S)\in \mathscr{U}$, which implies $\mu(S)=0$.
\end{proof}
The hypotheses of this result are known to be consistent with ZFC -- for instance, they hold true in Miller's model \cite{Miller2}. 

\smallskip

Let us now present the proof that uses forcing.

\smallskip
Let $\mathbb{B}$ be a Boolean algebra and $\mathcal{I}$ be an ideal (on $ \mathbb{B}$). 
Is there any criterion characterizing when  $\mathcal{I}$ is  a   measure ideal? A first preliminary observation is that $\mathcal{I}$ is a measure ideal if and only if there is a  measure $\bar{\mu}\colon \mathbb{B}/\mathcal{I}\rightarrow [0,1]$ such that $\ker(\bar{\mu})=\{[0]_{\mathcal{I}}\}$. 
Since the non-zero elements of $\mathbb{B}/\mathcal{I}$ identify naturally with  $\mathcal{I}^+:=\{b\in\mathbb{B}\mid b\notin \mathcal{I}\}$ (i.e., the $\mathcal{I}$-positive sets), the above tantamouts to  the existence of a strictly positive measure $\bar{\mu}\colon \mathcal{I}^+\rightarrow (0,1]$. Characterizing when (the set of non-zero elements of) a Boolean algebra supports such kind of measures has been a classical topic of research in topology \cite{StevoChanCond}.

\smallskip

A first attempt towards this characterization is by way of the countable chain condition. Recall that $\mathbb{B}$ has the \emph{countable chain condition} (in short, $\mathbb{B}$ \emph{is ccc}) if every set $X\s \mathbb{B}^+$ consisting of pairwise disjoint elements is countable. It is fairly easy to show that if $\mathbb{B}^+$ carries a strictly positive measure then $\mathbb{B}$ must be ccc. Nonetheless, there are classical examples showing that this implication is not reversible \cite{Gaifman}. 
A finer attempt passes through $\sigma$-bounded-cc Boolean algebras. We say that \emph{$\mathbb{B}$ is $\sigma$-bounded-cc} if $\mathbb{B}^+$ can be presented as $\bigcup_{n<\omega}\mathcal{B}_n$ where each $\mathcal{B}_n$ is $(n+1)$-cc, meaning that $\mathcal{B}_n$ contains at most $(n+1)$-many pairwise disjoint elements. Note that one does not loss any generality by assuming that $\mathcal{B}_n$ is $\leq$-upwards closed; to wit, if $b\in\mathcal{B}_n$ and $c\in\mathbb{B}^+$ is such that $b\leq c$ then $c\in\mathcal{B}_n$. If $\mathbb{B}^+$ supports a strictly positive measure $\mathcal{I}$ then $\mathbb{B}$ is clearly  $\sigma$-bounded-cc -- simply consider $\mathcal{B}_n:=\{b\in\mathbb{B}\mid \mu(b)\geq \frac{1}{n+1}\}.$

Once again, $\sigma$-bounded-ccness  is unsatisfactory at characterizing when $\mathbb{B}^+$ carries a strictly positive measure  \cite{Gaifman}. 
The final sought characterization is an important theorem due to Kelley \cite{Kelley}, which isolates the exact subfamily of $\sigma$-bounded-cc Boolean algebras  which admit a strictly positive measure. We refrain on delving into the specifics of Kelley's characterization in this paper for we want to kill all the $Q^+(\omega)$-measure ideals, and for this it suffices to kill all the $Q^+(\omega)$-$\sigma$-bounded-cc ideals:    
 
\begin{definition}
    An ideal $\mathcal{I}$ will be called \emph{$\sigma$-bounded-cc} provided $\mathcal{I}^+$ can be presented as a union $\bigcup_{n<\omega}\mathcal{I}^+_n$ with each $\mathcal{I}^+_n$ having the $(n+1)$-cc.
\end{definition}

We have pointed out in \S\ref{sec:prelimminaries} that a natural extension of $Q^+(\omega)$-ideals is given by rapid ideals. 
Recall that a function $f\colon X\to Y$ is \emph{finite-to-one} if $\{x\in X\mid f(x)=y\}$ is finite for every $y\in \range(f)$. The following facts are well known:
\begin{prop}\label{lemma:characterizingQpoints}
    The following are equivalent for an ideal $\mathcal{I}$:
    \begin{enumerate}
        \item $\mathcal{I}$ is a $Q^+(\N)$-ideal;
        \item For each finite-to-one function $f\colon \omega\rightarrow \omega$  there is $B\in\mathcal{I}^+$ and $k<\omega$ such that $f\restriction B$ is at most $k$-to-one. 
    \end{enumerate}
\end{prop}

\begin{prop}\label{lemma: equivalence rapid}
       The following are equivalent for an ideal $\mathcal{I}$:
    \begin{enumerate}
        \item $\mathcal{I}$ is rapid;
        \item  For every increasing function $f\colon \omega\rightarrow\omega$ there is  another increasing function $g\colon\omega\rightarrow\omega$  such that $f<g$ and $\{g(n): n<\omega\}\in\mathcal{I}^+$.\footnote{$f<g$ stands for the \emph{pointwise domination order}; to wit, $f(n)<g(n)$ for all $n<\omega.$}
    \end{enumerate}
\end{prop}




For this and other characterizations we refer the reader to \cite[Sec 4.6]{Bartoszynski}.
The goal of this subsection is to prove the following result:
\begin{theorem}\label{thm: no Q points}
    Assume $\gch$. Then there  is a cofinality-preserving generic extension in which no $\sigma$-bounded-cc ideal is  rapid. 
    
    In particular, no measure ideal is  $Q^+(\N)$.  
\end{theorem}

The above-mentioned generic extension is constructed via \emph{iterated forcing} \cite{Kunen}. Miller \cite{Miller} proved that in \emph{Laver's model for the Borel conjecture} \cite{Laver} there are no rapid ultrafilters (on $\mathcal{P}(\omega)$). Taking Miller's work as a stepping stone, we show that in Laver's and Mathias' models there are no rapid $\sigma$-bounded-cc ideals, and that in Miller's model there are no $\sigma$-bounded-cc $Q^+(\omega)$-ideals. All models are obtained after forcing over a model of ZFC+$\gch$ with a countably-supported iteration $$\mathbb{P}_{\omega_2}:=\varprojlim\langle \mathbb{P}_\alpha;\dot{\mathbb{Q}}_\alpha: \alpha<\omega_2\rangle$$ where each iterand $\dot{\mathbb{Q}}_\alpha$ is a $\mathbb{P}_\alpha$-name for either \emph{Laver's forcing} (see \cite[\S2]{Laver}), \emph{Mathias' forcing}, or \emph{Miller's forcing} (see \cite{Halbeisen}).  The exact definition of the respective posets is irrelevant for our intended purposes -- only acquaintance with  some of their key properties is needed. 

\smallskip

For the non-expert reader we recall a few generalities from iterated forcing theory used without reference later on. For more see \cite[\S5]{Kunen}.

For ordinals $1\leq \alpha\leq\beta\leq  \omega_2$ the map $\pi_{\beta,\alpha}\colon \mathbb{P}_\beta\rightarrow\mathbb{P}_\alpha$ given by $p\mapsto p\restriction\alpha$ defines a \emph{projection}. In particular, if $G$ is  $\mathbb{P}_{\beta}$-generic the filter generated by $\{p\restriction\alpha: p\in G\}$ (denoted $G_\alpha$) is $\mathbb{P}_\alpha$-generic. Conversely,  there are \emph{complete embeddings} $i_{\alpha,\beta}\colon \mathbb{P}_\alpha\rightarrow\mathbb{P}_\beta$ given by  $p\mapsto p'$ where   $p'=\langle p'_\gamma: \gamma<\beta\rangle\in \mathbb{P}_\beta$ stands for the unique condition with $p'\restriction\alpha=p$ and $p_\gamma=\dot{\one}_{\dot{\mathbb{Q}}_\gamma}$ for all $\gamma\in[\alpha,\beta).$\footnote{$\dot{\one}_{\dot{\mathbb{Q}}_\gamma}$ is the standard $\mathbb{P}_\gamma$-name for the trivial condition of $\dot{\mathbb{Q}}_\gamma$.} The map $\pi_{\beta,\alpha}$ extends naturally to $\mathbb{P}_\beta$-names  $\pi_{\beta,\alpha}\colon V^{\mathbb{P}_\beta}\rightarrow V^{\mathbb{P}_\alpha}$ by stipulating $\pi_{\beta,\alpha}(\sigma):=\{\langle \pi_{\beta,\alpha}(\tau),\pi_{\beta,\alpha}(p): \langle \tau,p\rangle\in \sigma\}$. Thereby every $\mathbb{P}_\beta$-name $\sigma$ can be ``projected'' by way of $\pi_{\beta,\alpha}$ to a $\mathbb{P}_\alpha$-name $\pi_{\beta,\alpha}(\sigma)$. In the same fashion the map $i_{\alpha,\beta}$ extends to $i_{\alpha,\beta}\colon V^{\mathbb{P}_\alpha}\rightarrow V^{\mathbb{P}_\beta}$ -- this permits to identify $\mathbb{P}_\alpha$-names as $\mathbb{P}_\beta$-names. The \emph{support of a condition $p\in\mathbb{P}_\kappa$} is the set defined as $\supp(p):=\{\gamma<\kappa: {p}_\gamma\neq \one_{\dot{\mathbb{Q}}_\gamma}\}$ and  $\mathbb{P}_\kappa$ is said to be \emph{countably-supported} if $\supp(p)$ is countable for all $p\in\mathbb{P}_\kappa$.
\begin{fact}\label{handy fact}
    Let  $G$ be $\mathbb{P}_\beta$-generic over $V$. For each $\mathbb{P}_\alpha$-name $\tau$,
    $$\tau_{G_\alpha}=(i_{\alpha,\beta}(\tau))_{G_\beta}.$$
\end{fact}

\smallskip

Let us now drive our attention to the proof of Theorem~\ref{thm: no Q points}.
\begin{Setup}\label{setup}
    As customary in forcing arguments we  denote our ground model by $V$ and we assume it satisfies $\gch$.  Working inside $V$, we let  
    $\mathbb{P}_{\omega_2}$  
    be the inverse limit of a countable support iteration $\langle \mathbb{P}_\alpha;\dot{\mathbb{Q}}_\alpha: \alpha<\omega_2\rangle$ such that: 
    \begin{enumerate}
        \item $\mathbb{P}_{\omega_2}$ is $\aleph_2$-cc.
        \item $\mathbb{P}_{\omega_2}$ preserves $\aleph_1$.
        \item $\mathbb{P}_{\omega_2}$ forces $``2^{\aleph_0}=\aleph_2$''.
        \item $\mathbb{P}_\alpha$ forces $``\ch\,\wedge\,\dot{\mathbb{Q}}_\alpha\text{ adds a set $x_\alpha\in \dot{\mathcal{P}}(\check{\omega})$ not in $\check{V}$}$'', for all $\alpha<\omega_2$.
    \end{enumerate}
    In what follows $G$ will be a fixed  $\mathbb{P}_{\omega_2}$-generic filter over $V$.
\end{Setup}

\begin{remark}
    It is known that countable support iterations of Laver, Mathias, and Miller's forcing all have properties (1)--(4) above. For details, see \cite{Laver, Halbeisen} or \cite[p.564]{Jech}.
\end{remark} 

\begin{lemma}[Catching-our-tail]\label{lemma: catch our tail}
    Suppose that $\mathcal{I}$ is a $\sigma$-bounded-cc ideal in $V[G]$. Then, there is an ordinal $\alpha<\omega_2$ such that $$V[G_\alpha]\models \text{$``\mathcal{I}\cap \mathcal{P}(\omega)^{V[G_\alpha]}$ is a $\sigma$-bounded-cc ideal''}.$$ 
\end{lemma}
\begin{proof}
         We establish this via a Löwenheim-Skolem-type argument. Working in the generic extension $V[G]$ we define a function $\psi\colon \omega_2\rightarrow\omega_2$ as follows.

    Fix $\alpha<\omega_2$ and enumerate the set $\mathcal{P}(\omega)^{V[G_\alpha]}$ as $\langle x^\alpha_\gamma: \gamma<\omega_1\rangle$. For each $\gamma<\omega_1$ let a  $\mathbb{P}_\alpha$-name $\dot{x}^\alpha_\gamma$ such that $x^\alpha_\gamma=(\dot{x}^\alpha_\gamma)_{G_\alpha}$. Since there is a natural way to regard $\dot{x}^\alpha_\gamma$ as a $\mathbb{P}_{\omega_2}$-name (via $i_{\alpha,\omega_2}$)  we will treat each  $\dot{x}^\alpha_\gamma$ as if they were already $\mathbb{P}_{\omega_2}$-names. Let us also fix  $\dot{\mathcal{I}}$ and $\{\dot{\mathcal{I}}^+_n: n<\omega\}$,  $\mathbb{P}_{\omega_2}$-names such that $\dot{\mathcal{I}}_G=\mathcal{I}$, and $\bigcup_{n<\omega}(\dot{\mathcal{I}}^+_n)_G$ witnessing that $\mathcal{I}$ is $\sigma$-bounded-cc in $V[G].$

    \smallskip

    Let $p_\alpha\in G$ be a condition in  $\mathbb{P}_{\omega_2}$ forcing the following sentence:
      $$\text{$``\langle \dot{x}^\alpha_\gamma: \gamma<\check{\omega}_1\rangle\s \dot{\mathcal{P}}(\omega)$ and $\textstyle\dot{\mathcal{I}}^+=\bigcup_{n<\omega}\dot{\mathcal{I}}^+_n\;\text{where each $\dot{\mathcal{I}}^+_n$ is $(n+1)$-cc}$''}.$$

      For each $\gamma<\omega_1$ and $n<\omega$  let  a maximal antichain $\mathcal{A}_{\alpha,\gamma,n}\s\mathbb{P}_{\omega_2}$ below $p_\alpha$ such that each  $q\in\mathcal{A}_{\alpha,\gamma,n}$  decides the sentence $``\dot{x}^\alpha_\gamma\in \dot{\mathcal{I}}^+_n$''.

\smallskip

      Given a maximal antichain $\mathcal{A}\s\mathbb{P}_{\omega_2}$ and an ordinal $\delta<\omega_2$ we say that \emph{$\delta$ is above the support of $\mathcal{A}$} if $\delta$ is above the support of all conditions $u\in\mathcal{A}$. For each $\gamma<\omega_1$ let $\delta_{\alpha,\gamma}<\omega_2$ be above the support of each $\mathcal{A}_{\alpha,\gamma,n}$  -- this is possible because $\mathbb{P}_{\omega_2}$ is $\aleph_2$-cc. Finally, let $\delta_\alpha>\max\{\alpha,\sup_{\gamma<\omega_1}\delta_{\alpha,\gamma}\}$ some ordinal below $\omega_2$ and stipulate $\psi(\alpha):=\delta_\alpha$. This yields $\psi\colon \omega_2\rightarrow\omega_2$.

      \smallskip

      Let $C$ be the club of closure points of $\psi$ (i.e., $C:=\{\alpha<\omega_2: \psi``\alpha\s \alpha\}$) and  $\alpha\in C$ be of cofinality $\omega_1$. Note that since $\mathbb{P}_{\omega_2}$ is cofinality preserving there is no ambiguity in this assertion. Define a $\mathbb{P}_\alpha$-name as follows:
      $$\dot{f}:=\{\langle\langle\check{n},\dot{f}_{n,r}\rangle, r \rangle: n<\omega,\;r\in\mathbb{P}_\alpha\},$$
      where $\dot{f}_{n,r}$ is the $\mathbb{P}_\alpha$-name defined as 
      $$\{\langle\dot{x},s \rangle: \text{$\dot{x}$ is a $\mathbb{P}_\alpha$-name, $s\leq r$ and $i_{\alpha,\omega_2}(s)\forces_{\mathbb{P}_{\omega_2}}\dot{x}\in \dot{\mathcal{I}}^+_n$}\}.$$

      \begin{claim}
          $\dot{f}_{G_\alpha}$ is a sequence with domain $\omega$.
      \end{claim}
      \begin{proof}[Proof of claim]
      Note that $\dom(\dot{f}_{G_\alpha})=\omega$ because each $n<\omega$ always appears inside $\dot{f}$ paired with a condition $r\in G_\alpha$. To show that indeed this is a sequence (i.e., a function) we have to show that $\dot{f}_{G_\alpha}(n)$ takes a unique value.

      Suppose that $\langle\langle\check{n},\dot{f}_{n,r}\rangle, r \rangle, \langle\langle\check{n},\dot{f}_{n,r'}\rangle, r' \rangle\in\dot{f}$ with $r,r'\in G_\alpha$. We will prove (using a double-inclusion argument) that $(\dot{f}_{n,r})_{G_\alpha}=(\dot{f}_{n,r'})_{G_\alpha}$. 
      
      Let $x\in (\dot{f}_{n,r})_{G_\alpha}$. By definition, there is $\langle \dot{x},s\rangle\in\dot{f}_{n,r}$ such that $x=\dot{x}_{G_\alpha}$ and $s\in G_\alpha$. Since $s, r'\in G_\alpha$ one finds $s'\leq s,r'$ in $G_\alpha$. A moment's reflection makes clear that $\langle \dot{x},s'\rangle\in\dot{f}_{n,r'}$, which yields $x=\dot{x}_{G_\alpha}\in(\dot{f}_{n,r'})_{G_\alpha}$. The converse inclusion is established in the very same fashion.
      \end{proof}
      The following is the key observation:
      \begin{claim}
          $\dot{f}_{G_\alpha}=\langle \mathcal{I}^+_n\cap\mathcal{P}(\omega)^{V[G_\alpha]}: n<\omega\rangle.$
      \end{claim}
      \begin{proof}[Proof of claim]
        Let us prove the claim by double inclusion. Fix $n<\omega$.

\smallskip

$\boxed{\subseteq}:$ Suppose that $x\in\dot{f}_{G_\alpha}(n)$. That means that $x\in (\dot{f}_{n,r})_{G_\alpha}$ for some (equivalently, all) $r\in G_\alpha$. By definition, there is $\langle \dot{x},s\rangle\in\dot{f}_{n,r}$ with $s\in G_\alpha$ and $x=\dot{x}_{G_\alpha}$.  Since $i_{\alpha,\omega_2}(s)\in G$ and $i_{\alpha,\omega_2}(s)\forces_{\mathbb{P}_{\omega_2}}\dot{x}\in\dot{\mathcal{I}}^+_n$ we infer that $x\in \mathcal{I}^+_n\cap\mathcal{P}(\omega)^{V[G_\alpha]}$, which disposes the first inclusion.

\smallskip

$\boxed{\supseteq}$ :  Conversely, let $x\in\mathcal{I}^+_n\cap \mathcal{P}(\omega)^{V[G_\alpha]}$.  Since $\cf(\alpha)=\omega_1$ there must be $\bar{\alpha}<\alpha$ such that $x\in\mathcal{P}(\omega)^{V[G_{\bar{\alpha}}]}$.\footnote{For each $n<\omega$ let $q_n\in G_\alpha$ deciding $``n\in\dot{x}$''. Since all conditions have countable support and $\cf(\alpha)=\omega_1$ there is $\bar\alpha<\alpha$ such that all conditions have support below $\bar\alpha$. Define the $\mathbb{P}_{\bar{\alpha}}$-name $\dot{y}=\{\langle \check{n}, q\rangle:  q\in \mathbb{P}_{\bar\alpha}\, \wedge\,i_{\bar\alpha,\alpha}(q)\forces_{\mathbb{P}_\alpha}\check{n}\in \dot{x}\}$. One can show that $\dot{x}_{G_\alpha}=\dot{y}_{G_{\bar\alpha}}$.} Let $\gamma<\omega_1$ be such that $x=(\dot{x}^{\bar{\alpha}}_\gamma)_{G_{\bar{\alpha}}}$ (recall that $\dot{x}^{\bar\alpha}_\gamma$ was a $\mathbb{P}_{\bar\alpha}$-name). In our construction we identified a maximal antichain $\mathcal{A}_{\bar{\alpha},\gamma,n}$ of conditions $q$ such that $q\parallel_{\mathbb{P}_{\omega_2}}``\dot{x}^{\bar{\alpha}}_\gamma\in\dot{\mathcal{I}}^+_n$''. Since $\alpha$ was a closure point of $\psi$ and $\bar{\alpha}<\alpha$ it follows  that $\alpha$ is above the support of  $\mathcal{A}_{\bar{\alpha},\gamma,n}$. Thus, $q=i_{\alpha,\omega_2}(\pi_{\omega_2,\alpha}(q))$  for all $q\in\mathcal{A}_{\bar\alpha,\gamma,n}$. A careful inspection of the name $\dot{f}_{n,\pi_{\omega_2,\alpha}(q)}$ should convince the reader that $$\langle i_{\bar\alpha,\alpha}(\dot{x}^{\bar{\alpha}}_\gamma), \pi_{\omega_2,\alpha}(q)\rangle\in \dot{f}_{n,\pi_{\omega_2,\alpha}(q)}$$ for all $q\in\mathcal{A}_{\bar{\alpha},\gamma,n}$. Letting $q^*$ be the unique condition in $G\cap \mathcal{A}_{\bar{\alpha},\gamma,n}$ we have that $\pi_{\omega_2,\alpha}(q^*)\in G_\alpha$ and thus $i_{\bar\alpha,\alpha}(\dot{x}^{\bar{\alpha}}_\gamma)_{G_{{\alpha}}}\in (\dot{f}_{n,\pi_{\omega_2,\alpha}(q)})_{G_\alpha}$. By using Fact~\ref{handy fact} we have $x\in (\dot{f}_{n,\pi_{\omega_2,\alpha}(q)})_{G_\alpha}=\dot{f}_G(n)$, as needed.
\end{proof}
In particular, $\langle \mathcal{I}^+_n\cap \mathcal{P}(\omega)^{V[G_\alpha]}: n<\omega\rangle$ belongs to $V[G_\alpha]$ and clearly each $\mathcal{I}^+_n\cap \mathcal{P}(\omega)^{V[G_\alpha]}$ is $(n+1)$-cc (by absoluteness). Since $\mathcal{I}^+$ decomposes as $\bigcup_{n<\omega}\mathcal{I}^+_n$ it is clear that $\mathcal{I}^+\cap \mathcal{P}(\omega)^{V[G_\alpha]}$ can be presented as the union of the $\mathcal{I}^+_n\cap \mathcal{P}(\omega)^{V[G_\alpha]}$'s. Certainly, this implies that $\mathcal{I}\cap \mathcal{P}(\omega)^{V[G_\alpha]}$  is $\sigma$-bounded-cc in the intermediate model $V[G_\alpha].$
      \end{proof}

In the context of the forthcoming arguments $\mathbb{P}_{\omega_2}$ will denote either the Laver or the Miller iteration defined in \cite{Laver} and \cite{Miller2}, respectively. (Note that in either case $\mathbb{P}_{\omega_2}$ fulfills the assumptions described in Setup \ref{setup}.)

\begin{lemma}\label{lemma: Laverlike}\hfill
 \begin{enumerate}
     \item (\cite{Miller}) Suppose that $\mathbb{P}_{\omega_2}$ is the Laver iteration.   Let $p\in\mathbb{P}_{\omega_2}$ and $\tau$ and $\dot{\ell}$ be $\mathbb{P}_{\omega_2}$-names for an element of $\omega^\omega$ and the first  real introduced by $\mathbb{P}_{\omega_2}$, respectively. Suppose that $p$ forces the following sentence:
$$\text{$``\tau$ is an increasing function in  $\dot{\omega}^\omega$ and $\dot{\ell}<\tau$''.}$$
Then, there are conditions  $q_0,q_1$ and sets $Z_0,Z_1$ such that $q_i\leq p$, $Z_i\cap Z_j$ is finite and $q_i\forces_{\mathbb{P}_{\omega_2}}\tau``\check{\omega}\s \check{Z}_i$ for all $i<2.$
\item (\cite{Miller2}) Suppose that $\mathbb{P}_{\omega_2}$ is the Miller iteration. Let $p\in\mathbb{P}_{\omega_2}$ and $\dot A$ and $\dot{\ell}$ be $\mathbb{P}_{\omega_2}$-names for a real and the first real introduced by $\mathbb{P}_{\omega_2}$ (which is assumed to be increasing), respectively.
          Suppose that $p$ forces the following sentence:
$$\forall n \in \omega, \, \lvert [\dot{\ell}(n),\dot{\ell}(n+1)) \cap A\lvert \leq 1 $$
Then, there are conditions $q_0,q_1$ and sets $Z_0,Z_1$ such that $q_i\leq p$, $Z_i\cap Z_j$ is finite and $q_i\forces_{\mathbb{P}_{\omega_2}} \dot A\s \check{Z}_i$ for all $i<2.$\qed
 \end{enumerate}
\end{lemma}

\begin{remark}
  A similar result applies to $\omega_2$-length countable support iterations of \emph{Mathias forcing}. Please refer to \cite[Lemma~25.9]{Halbeisen}.  
\end{remark}

The following is a spinoff of \cite[Lemma~2]{Miller} (Clause (1) of Lemma~\ref{lemma: Laverlike}) and \cite[Lemma~25.9]{Halbeisen}. The analogue parallel result for \cite[Claim 5.1.1]{Miller2} can be proven by way of the same line of reasoning.

\begin{lemma}\label{lemma: miller}
Suppose that $\mathbb{P}_{\omega_2}$ is Laver's/Miller's iteration and let $p,\tau,\dot{\ell}$ be as in the previous lemma.
Then, for each $k\geq 1$ there are conditions $\{q_i: i<2^k\}$ and sets $\{Z_i: i<2^k\}$ such that $q_i\leq p$, $Z_i\cap Z_j$ is finite and $q_i\forces_{\mathbb{P}_{\omega_2}}\tau``\check{\omega}\s \check{Z}_i$ for all $i,j<2^k.$
\end{lemma}
\begin{proof}
   The case $k=1$ is exactly our departing hypothesis. Suppose our claim were valid for $k$ and let us prove it for $k+1$. For each $i<2^k$ apply the base case $k=1$ and find $q^0_{i}$ and $q^1_{i}$  and sets $Z_{i}^0$ and  $Z_{i}^1$. Without losing any generality we may assume that $Z^0_i, Z^1_i\s Z_i$ -- otherwise take their intersections with $Z_i$ (note that this is alright in that $q^x_i\leq q_i$ and  $q_i^x\forces_{\mathbb{P}_{\omega_2}}``\tau``\omega\s \check{Z}^x_i$''). By re-enumerating all of these objects  appropriately we obtain conditions $\{q^*_i: i<2^{k+1}\}$ and sets $\{Z^*_i: i<2^{k+1}\}$ with the desired properties.
\end{proof}

The time is ripe to prove  the section's main result, Theorem~\ref{thm: no Q points} (see p.\pageref{thm: no Q points}). 
\begin{proof}[Proof of Theorem~\ref{thm: no Q points}]
Let $\mathbb{P}_{\omega_2}$ be Laver/Miller's iteration and $G\s \mathbb{P}_{\omega_2}$ generic over $V$.
    Let $\mathcal{I}$ be a $\sigma$-bounded-cc ideal in $V[G]$ and let us show that it cannot be rapid. By Lemma~\ref{lemma: catch our tail} there is a stage $\alpha<\omega_2$ such that $\mathcal{I}\cap\mathcal{P}(\omega)^{V[G_\alpha]}$ is a $\sigma$-bounded-cc ideal in $V[G_\alpha].$ Working inside $V[G_\alpha]$, we will apply Lemma~\ref{lemma: miller} to the $V[G_\alpha]$-version of the iteration $\mathbb{P}_{\omega_2}$ (which  is forcing equivalent to the quotient forcing  $\mathbb{P}_{\omega_2}/G_\alpha:=\{p\in\mathbb{P}_{\omega_2}: p\restriction\alpha\in G_\alpha\}$). 
    Our claim is that $\mathcal{I}_\alpha:=\mathcal{I}\cap\mathcal{P}(\omega)^{V[G_\alpha]}$ does not have any extension (by way of $\mathbb{P}_{\omega_2}/G_\alpha$) to a rapid ideal. Moreover, we claim that this is witnessed by $\ell_\alpha$, the first real introduced (over $V[G_\alpha]$) at stage $\alpha$ (see Lemma~\ref{lemma: equivalence rapid}). Suppose towards a contradiction that $g\colon \omega\rightarrow\omega$ is  an  increasing function in $V[G]=V[G_\alpha][G]$, that dominates $\ell_\alpha$ and that $\{g(n): n<\omega\}$ is $\mathcal{I}$-positive. Let $\{\mathcal{I}^+_n: n<\omega\}$ witness for $\mathcal{I}$ being $\sigma$-bounded-cc and  $k<\omega$ with $g``\omega\in\mathcal{I}^+_k$.

    \smallskip

    Let $\dot{\ell}_\alpha$, $\dot{g}$ and $\dot{\mathcal{I}}^+_k$ be $\mathbb{P}_{\omega_2}/G_\alpha$-names and a condition $p\in G$ such that
    $$(\star)\;\;V[G_\alpha]\models p\forces_{\mathbb{P}_{\omega_2}/G_\alpha}``\dot{\ell}_\alpha<\dot{g}\,\wedge\,\dot{g}\,\text{is increasing $\wedge$ $\dot{g}``\check{\omega}\in\dot{\mathcal{I}}^+_k$''.}$$

 Working inside $V[G_\alpha]$ define 
       $$D:=\{q\leq p: \exists Z\;(Z\notin (\mathcal{I}^+_{k}\cap \mathcal{P}(\omega))\,\wedge\, q\forces_{\mathbb{P}_{\omega_2}/G_\alpha}\dot{g}``\check{\omega}\s \check{Z})\}.$$
       For this to be definable in $V[G_\alpha]$ we crucially use $\mathcal{I}^+_k\cap \mathcal{P}(\omega)^{V[G_\alpha]}\in V[G_\alpha]$. 
    \begin{claim}
   $D$ is dense below $p$.
    \end{claim}
    \begin{proof}[Proof of claim]
        Let $q\leq p$ be a condition in $\mathbb{P}_{\omega_2}/G_\alpha$. Apply Lemma~\ref{lemma: miller} inside $V[G_\alpha]$ to $q$, the names $\dot{\ell}_\alpha,\dot{g}$ and the integer $k$. This way we obtain conditions $\{q_i: i<2^k\}$ and sets $\{Z_i: i<2^k\}$ such that $q_i\leq q$, $Z_i\cap Z_j$ is finite and $q_i\forces_{\mathbb{P}_{\omega_2}/G_\alpha}\dot{g}``\check{\omega}\s \check{Z}_i$ for all $i,j<2^k.$ Since the intersections  $Z_i\cap Z_j$ are $\mathcal{I}_\alpha$-null and $\mathcal{I}^+_{k}\cap \mathcal{P}(\omega)^{V[G_\alpha]}$ is $(k+1)$-cc there must be  $i^*<2^k$ such that $Z_{i^*}\notin \mathcal{I}^+_{k}\cap \mathcal{P}(\omega)^{V[G_\alpha]}$. All in all, $q_{i^*}\leq q$ and $q_{i^*}\in D$.
    \end{proof}
    By density and since $p\in G$ there is $q\in D\cap G$. This yields a contradiction: First, $q$ forces $(\star)$, hence $g``\omega\in\mathcal{I}^+_k$; second,  $q\in D$ so by definition there is $Z$ not in $\mathcal{I}^+_k\cap \mathcal{P}(\omega)$ such that $g``\omega\s Z$. Note that this is impossible because $\mathcal{I}^+_k$ is upwards $\s$-closed.
\end{proof}

As an immediate consequence of the above results we obtain the following quotable corollary:

\begin{cor}\label{cor: final}
 In Laver's model \cite{Laver} or in the Mathias model  \cite[pp. 407--415]{Halbeisen}  there are no rapid $\sigma$-bounded-cc ideals.  In particular, in those models and in Miller's model \cite{Miller2} there are no measure  $Q^{+}(\N)$-ideals. 
\end{cor}

\section{A Banach space with $L$-orthogonal sequences but without $L$-orthogonal elements}\label{sec:BanachSpaceLorthogonal}


\smallskip

\subsection{Strong $Q$-measures and the second dual}\label{strongstarstar}

Recall that a family $\mathcal{H}\subseteq \mathcal{P}(\N)$ is called \emph{$\omega$-hitting} if given any sequence $(A_n)_{n\in \N}$ of infinite subsets of $\N$ there is $A\in \mathcal{H}$ such that for any $n\in \N$, $A_n\cap A$ is infinite. We generalize\footnote{\cite[Theorem 3.3]{Hrusak} claims that for any Borel ideal $\mathcal{I}$ either $\non^*(\mathcal I)= \omega$ or $\mathcal{ED}_{fin}\leq_{KB}\mathcal I$. Now, $\non^*(\mathcal I)\leq\omega$ if, and, only if $\mathcal I$ is not $w$-hitting. Also, $\mathcal{ED}_{fin}\leq_{KB}\mathcal I$ if, and only if, there is a partition of $\omega$ into finite sets $\mathcal P$, such that $\mathcal I$ contains all the selectors of $\mathcal{P}$. So, the following theorem generalizes the cited result.} \cite[Theorem 3.3]{Hrusak} as: 

\begin{theorem}[Hru\v{s}\'{a}k-Meza-Minami]
Let $\mathcal{H}$ be an analytic hereditary family. Then the following are equivalent:
\begin{enumerate}
    \item For any partition of $\N$ into finite sets $\mathcal{P}$, $\mathcal{H}$ does not contain all the selectors of $\mathcal{P}$.
    \item $\mathcal{H}$ is not $\omega$-hitting.\qed
\end{enumerate}
\end{theorem}
\begin{proof}
 The implication $(2)\implies(1)$ always hold and can be shown by a straightforward argument, so we focus on the implication $(1) \implies (2)$.  
 First appeal to the result in \cite{Spinas} that claims that any analytic $\omega$-splitting\footnote{$\mathcal{A}\subseteq [\omega]^\omega$ is $\omega$-splitting if for any countable $\mathcal X\subseteq [\omega]^\omega$ there is $A\in \mathcal{A}$ such that for any $X\in \mathcal{X}$, $\lvert A\cap X\lvert = \lvert X\setminus A \lvert = \omega$.} family contains a closed $\omega$-splitting family, then notice that any $\omega$-hitting hereditary family is $\omega$-splitting (just split the witness of being $\omega$-hitting). So any analytic $\omega$-hitting hereditary family contains a Borel hereditary $\omega$-hitting family. This reduces the problem to the Borel case.

In order to prove the result in the case $\mathcal{H}$ is Borel it is enough to reproduce the original proof in \cite[Theorem 3.3]{Hrusak}, because it only uses the hereditary property of ideals.
\end{proof}

\begin{cor}\label{epsilon measures corolary}
   Let $\mu\colon \mathcal{P}(\omega) \rightarrow [0,1]$ be a measure and $\varepsilon>0$. Then, the following assertions are equivalent:
   \begin{enumerate}
       \item $\mu$ is an $\varepsilon$-strong $Q$-measure.
       \item For every $F_\sigma$ hereditary family $\mathcal{H}$ that is $\omega$-hitting,  there is $A\in \mathcal{H}$, such that $\mu(A)\ge\varepsilon$.
       \item For every analytic hereditary family $\mathcal{H}$ that is $\omega$-hitting,  there is $A\in \mathcal{H}$, such that $\mu(A)\ge\varepsilon$.
   \end{enumerate}
\end{cor}

\begin{proof}
    $(1)\Rightarrow (3)$. Let $\mathcal{H}$ be an $\omega$-hitting hereditary analytic family. It follows from the previous theorem that there exists a partition $\mathcal{P}$ of $\N$ into finite sets such that $\mathcal{H}$ contains all the selectors of $\mathcal{P}$.
    Since $\mu$ is an $\varepsilon$-strong Q-measure, there exists $A\in \mathcal{H}$ such that $\mu(A)\geq \varepsilon$.

    $(3)\Rightarrow (2)$. This is trivial.
    
    $(2)\Rightarrow (1)$. Let $\mathcal{P}$ be an arbitrary partition of $\N$ into finite sets. Notice that the set of selectors of $\mathcal{P}$ is $F_\sigma$, hereditary and $\omega$-hitting. Then, there must be a selector $A$ such that $\mu(A)\ge \varepsilon$.
\end{proof}

Let $X$ be a Banach space, $(x_n)_{n\in \omega}$ a sequence in $B_X$, $(\varepsilon_n)_{n\in \omega}$ a sequence of positive real numbers converging to zero, $Z\subseteq X$ a separable subspace of $X$, and $(F_n)_{n\in \omega}$ an increasing sequence of finite dimensional subspaces of $X$ such that $Z=\overline{\bigcup_{n\in \omega}F_n}$. For any $B\subseteq \omega$ and $n\in \omega$, denote:
\begin{itemize}
    \item $B(n)$ the {$n$-th element} of $B$;
    \item $\overrightarrow{B(n)}=\{A\subseteq B: \min A\ge B(n)\}$;
    \item $C[B]:=\{w\in X: w\in \overline{\con}\{x_m:m\in B\}\}$.
\end{itemize}
Define the set $\mathcal{L}_{(F_n)_{n\in \omega}}$ as
    $$\{B\subseteq \omega: \forall n \in \omega \,\forall  A\in \overrightarrow{B(n)} \,\forall w\in C[A] \, \forall y\in F_n~~ (1-\varepsilon_n) (1+\norm{y})\le \norm{y+w} \}.$$


\begin{theorem} \label{Theorem strong Q-points}
Let $X$ be a Banach space, $(x_n)_{n\in \omega}$ an L-orthogonal sequence, and $\mu\colon \mathcal{P}(\omega) \rightarrow [0,1]$ a strong $Q$-measure with $\mu(\omega)=1$. Then  $\mu\text{-}\lim x_n$ is an L-orthogonal element.\footnote{The notation $\mu\text{-}\lim x_n$, related to the operator $T$, was presented in Notation \ref{notation:specialnotation}.}
\end{theorem}

\begin{proof}
Consider $z\in X$, $F_n = span\{z,x_0,\ldots,x_n\}$, and $(\varepsilon_n)_{n\in \omega}$ a sequence of positive reals converging to zero. It 
is clear that $\mathcal{L}_{(F_n)_{n\in \omega}}$ is $F_\sigma$ (indeed closed) and hereditary. As a consequence of \cite[Theorem 5.4]{HrusakSaenz}, it is also $\omega$-hitting. By Corollary \ref{epsilon measures corolary} above there is $B\in\mathcal{L}_{(F_n)_{n\in \omega}}$ such that $\mu(B)=1$. Since $\mu\colon \mathcal{P}(\omega)\rightarrow [0,1]$ vanishes on finite sets, it satisfies $\mu(\omega)=\mu(B)=\mu(\{j \in B: j>n\})=1$ for every $n$. It follows that $$\mu\in \bigcap_n \overline{\mathrm{conv}}^{w^{*}}(\{e_j:j\in B,~j>n\})\subseteq \ell_1^{**} = \ell_\infty^{*},$$ and hence
$$T^{**}(\mu) = \mu\text{-}\lim x_n\in \bigcap_n \overline{\mathrm{conv}}^{w^{*}}(\{x_j:j\in B,~j>n\})\subset X^{**},$$ 
where $T\colon \ell_1 \rightarrow X$ is the unique operator satisfying $Te_n=x_n$ for every $n$.
It is  clear that $(x_n)_{n\in B}$ is $L$-orthogonal, so by \cite[Lemma 3.3]{AvilesMartinezRueda} $$\norm{T^{**}(\mu)+z}=\norm{T^{**}(\mu)}+\norm{z}=1+\norm{z}.$$
This completes the verification.
\end{proof}

\subsection{A space for each fit $Q$-measure}\label{sec: A space for each $Q$-measure}

In this section we construct, for each measure $\mu$ that is not a fit $Q$-measure, a space $X$ and an $L$-orthogonal sequence $(x_n)$ such that $\mu\text{-}\lim x_n$ is not an $L$-orthogonal element.

\begin{definition}
    Let $\mathcal P=\{P_n: n<\omega\}$ be a partition of $\omega$ into finite sets. 

    \begin{enumerate}
        \item  The \emph{ideal associated to the partition $\mathcal{P}$} is
    $$\mathcal{I}_{\mathcal P}=\{B\subseteq \omega: \exists k< \omega\; \forall n<\omega\; \lvert P_n\cap B\lvert\leq k \}=\{B\subseteq \omega: \sup_n \lvert P_n\cap B\lvert < \infty \}.$$
    \item The \emph{space associated to $\mathcal{P}$} is 
    $$K_{\mathcal P}=\{x\in \{-1,1\}^{\omega}: |\{ n\in P_m : x(n)=-1\}|\leq 1 \mbox{ for every }m<\omega\}.$$
    \end{enumerate}
\end{definition}

\
    The space $K_{\mathcal{P}}$ is a compact space  homeomorphic to the Cantor set, cf. \cite[Proof of Proposition 7.2]{AvilesMartinezRueda}.

\begin{notation}
   For each $n< \omega$ we shall denote by $f_n$ the evaluation map at coordinate $n$ for elements in $K_{\mathcal{P}}$; namely, the map $f_n(x) = x(n)$.
\end{notation}



From \cite[Proposition 7.2]{AvilesMartinezRueda} it is known that $(f_n)_{n\in\omega}$ is an $L$-orthogonal sequence $\mathcal{I}_\mathcal{P}$-converging in the weak topology to $\mathbbm{1}$, the function with constant value $1$.

\smallskip


In what follows we will bear on the usual identification $C(K_{\mathcal P})^*=M(K_{\mathcal P})$ via \emph{Riesz representation theorem}, where $M(K_{\mathcal P})$ denotes the space of Radon measures over $K_{\mathcal P}$ endowed with the norm of total variation. 


\begin{lemma} \label{theorem on lower from delta}
    Suppose that $\mu$ is a measure for which there is $\delta\ge 0$ and $\mathcal{P}$ a partition such that for any $I\in \mathcal{I}_{\mathcal P}$, $\mu(I)\leq\delta\norm{\mu}$. Then, $$\norm{\mu\text{-}\lim f_n - \norm{\mu}\mathbbm {1}}\leq2\delta\norm{\mu}.$$ 
\end{lemma}

\begin{proof} Assume $\mu$ is normalized. 
    Let $\mathcal P$ and $\delta>0$  be as in the hypothesis. Consider $T$ the operator used to define the $\mu\text{-}\lim$. For each $\nu\in M(K_{\mathcal P})$ of norm less than or equal to 1, 
$$T^{**}(\mu)(\nu)=\mu(T^*(\nu))=\mu((\nu(f_n))_{n}).$$

Let $\varepsilon>0$. To conclude the proof it is enough to show that 
$$\lvert \mu ((\nu(f_n))_{n\in \N}))-\nu(\mathbbm 1)\lvert\leq \varepsilon +2\delta.$$
We know $\mathcal (\mathcal{I}_{\mathcal P})^*$-$\lim f_n=\mathbbm 1$ in the weak topology.  Therefore, if $F_\varepsilon:=\{n<\omega: \lvert\nu(f_n)-\nu(\mathbbm 1)\lvert<\varepsilon\}$ then $F_\varepsilon \in (\mathcal I_{\mathcal P})^*$. By hypothesis this means $\mu(\omega\setminus F_\varepsilon)\leq\delta$.
So, $$\lvert \mu ((\nu(f_n))_{n\in \N}))-\nu(\mathbbm 1)\lvert\leq$$$$\lvert\mu((\nu(f_n)-\nu(\mathbbm 1))_{n\in \omega}\cdot\rchi_{ F_\varepsilon})\lvert+ \lvert \mu((\nu(f_n)-\nu(\mathbbm 1))_{n\in \omega}\cdot\rchi_{ \omega\setminus F_\varepsilon})\lvert,$$
where $\rchi_A$ is the characteristic function of $A$, and the multiplication is on each coordinate.
Since $\|(\nu(f_n)-\nu(\mathbbm 1))_{n\in \omega}\cdot\rchi_{ F_\varepsilon}\|_\infty \leq \varepsilon$, we have that $$\lvert\mu((\nu(f_n)-\nu(\mathbbm 1))_{n\in \omega}\cdot\rchi_{ F_\varepsilon})\lvert \leq \varepsilon.$$ 
Moreover, since $\|f_n-\mathbbm 1\| = 2$ for every $n$ and $\mu(\omega\setminus F_\varepsilon)\leq\delta$, we have that
$$\lvert \mu((\nu(f_n)-\nu(\mathbbm 1))_{n\in \omega}\cdot\rchi_{ \omega\setminus F_\varepsilon})\lvert\leq 2\delta,$$
which finishes the proof.
\end{proof}


\begin{definition}
For $\delta>0$, a measure $\mu\colon \mathcal{P}(\omega)\to [0,+\infty)$ is called a \emph{$\delta$-fit $Q$-measure} if for any $\{A_n\mid n\in \omega\}$ partition of $\omega$ into finite sets there is $A\subseteq \omega$ and $k\in \omega$ such that $\mu(A)\ge\delta \mu(\omega)$ and $\lvert A_n \cap A \lvert \leq k$ for any $n\in \omega$.  Those measures that are $\delta$-fit $Q$-measures for all $\delta<\frac{1}{2}$ are called fit $Q$-measures.
\end{definition}

The following result will play a key role in the next subsection.

\begin{lemma}\label{Ready for Main teo}
    Let $x^{**}\in S_{\ell_1^{**}}=S_{\ell_\infty^*}$ be a measure whose positive and negative parts are not fit $Q$-measures. Then, there is a bounded operator $T:\ell_1\to C(K)$, $\|T\|= 1$, with $K$ homeomorphic to the Cantor set, and a constant $c\in \mathbb{R}$ such that $(Te_n)_{n\in \omega}$ is an $L$-orthogonal sequence but $$\|T^{**}x^{**}-c\mathbbm{1}\|< \|x^{**}\|,$$
    so $T^{**}x^{**}$ is not $L$-orthogonal.
\end{lemma}
\begin{proof} 
We first consider the case in which $x^{**}$ does not vanish on all finite sets of $\omega$. In that case there exists $n\in \omega$ such that $x^{**}(\{n\})\neq 0$. Then $x^{**}=x^{**}(\{n\})e_n+y^{**}$, where $y^{**}\in \ell_1^{**}$ satisfies $$\|x^{**}\|=|x^{**}(\{n\})|+\|y^{**}\|.$$
Take $K=\{-1,1\}^\omega$, $T\colon \ell_1 \longrightarrow C(K)$ given by $T(e_k)=f_k$ whenever $k \neq n$ and $T(e_n)=0$, and $c=0$.
Then it is immediate that $\|T\|=1$,  $(Te_n)_{n\in \omega}$ is an $L$-orthogonal sequence but $$\|T^{**}x^{**}\|=\|T^{**}(x^{**}(\{n\})e_n+y^{**})\|=\|T^{**}(y^{**})\|\leq \|y^{**}\|< \|x^{**}\|$$
as desired.

Thus, from now on we suppose that $x^{**}$ vanishes on all finite sets of $\omega$. We identify $\ell_1^{**} = C(\beta\omega)^*$ as the space of signed regular Borel measures on $\beta\omega$ and consider the Hahn decomposition into positive and negative part of a signed measure, $x^{**} = \mu^+ - \mu^-$ so that $\mu^+$ and $\mu^-$ are supported inside disjoint Borel subsets $\Omega_+$ and $\Omega_-$ of $\beta\omega$. By hypothesis, $\mu^+$ and $\mu^-$ fail to be fit $Q$-measures. This means that there exists $\delta<\frac{1}{2}$ and partitions $\mathcal{P}^+$ and $\mathcal{P}^-$ of $\omega$ into finite sets such that $\mu^{\star}(I)<\delta\|\mu^\star\|$ for every $I\in \mathcal{I}_\mathcal{P^\star}$ and $\star\in \{+,-\}$. Fix $\varepsilon>0$ such that $\delta+\varepsilon<\frac{1}{2}$. By the regularity of the measures we can find closed sets $A_+ \subset \Omega_+$ and $A_- \subset \Omega_-$ such that 
$$\mu^+(\beta\omega\setminus A_+) < \varepsilon \|\mu^+\|,\ \  \mu^-(\beta\omega\setminus A_-) <\varepsilon\|\mu^-\|.$$
We can find a decomposition of $\beta\omega$ into two disjoint clopen sets $\beta\omega = L_+\cup L_-$ such that $A_+\subset L_+$ and $A_-\subset L_-$. They correspond to a partition $\omega = B_+\cup B_-$ of $\omega$. Define a new partition 
$$ \mathcal{P} = \{ P\cap B_+ : P\in \mathcal{P}^+\} \cup \{ P\cap B_- : P\in \mathcal{P}^- \},$$
together with the corresponding $K=K_\mathcal{P}$ (recall that $K_{\mathcal{P}}$ is homeomorphic to the Cantor set), functions $f_n = Te_n$ and operator $T\colon \ell_1\rightarrow C(K)$ as before.
If $I\in\mathcal{I}_\mathcal{P}$, then
$$\mu^+(I) = \mu^+(I\cap B_+) + \mu^+(I\cap B_-) \leq \delta \|\mu^+\| + \varepsilon\|\mu^+\| = (\delta+\varepsilon)\|\mu^+\|, $$
and similarly
$$\mu^-(I) \leq  (\delta+\varepsilon)\|\mu^-\|.$$

 By Lemma~\ref{theorem on lower from delta}, we have 
 $$\norm{\mu^+\text{-}\lim f_n - \norm{\mu^+}\mathbbm {1}}\leq2(\delta+\varepsilon)\norm{\mu^+},$$
  $$\norm{\mu^-\text{-}\lim f_n - \norm{\mu^-}\mathbbm {1}}\leq2(\delta+\varepsilon)\norm{\mu^-}.$$

 Remember that $\nu\text{-}\lim f_n = T^{**}\nu $ by definition, so adding up and using triangle inequality, we get 
  $$\norm{T^{**}x^{**}- (\norm{\mu^+}-\norm{\mu^-})\mathbbm {1}}\leq2(\delta+\varepsilon)\norm{x^{**}} < \norm{x^{**}}.$$
\end{proof}
\begin{remark}
 In the model of Theorem~\ref{thm: no Q points} where no $Q^+(\omega)$-measures exist any element of $S_{\ell_1^{**}}$ fulfills the above-displayed hypothesis.     
\end{remark}

Let us summarize what we know. Given a positive measure $\mu$, to ensure that any $\mu\text{-}\lim$ through an $L$-orthogonal sequence is an $L$-orthogonal element it is \emph{sufficient} for $\mu$ to be \emph{strong} $Q$. To ensure that the $\mu\text{-}\lim$ is $L$-orthogonal it is \emph{necessary} for $\mu$ to be \emph{fit} $Q$. But we do not know which property, if any, characterizes that the $\mu\text{-}\lim$ is $L$-orthogonal.

\subsection{A space without $L$-orthogonal elements}\label{sec: A space without $L$-orthogonal elements}

The goal of this subsection is to complete the proof of Theorem~\ref{Main Main theorem} from the introduction to this paper. The more precise formulation, after Theorem~\ref{thm: no Q points}, is the following:

\begin{theorem}\label{theo:sucesinLorto}
Suppose no fit $Q$-measures exist. Then there is a Banach space $X$ with an $L$-orthogonal sequence whose bidual $X^{**}$ does not carry  $L$-orthogonal elements.
\end{theorem}

 We commence with a series of auxiliary lemmata about $L$-orthogonal elements which will set the grounds to tackle  Theorem~\ref{theo:sucesinLorto}.  Given a (non-empty) indexed family of Banach spaces $\mathcal{X}=\{X_i: i\in\mathcal{I}\}$ we  denote by $(\bigoplus_{i\in\mathcal{I}}X_i)_\infty$ the \emph{$\ell_\infty$-sum of $\mathcal{X}$}; to wit, the Banach space consisting of sequences $(x_i)_{i\in\mathcal{I}}\in \prod_{i\in\mathcal{I}}X_i$ such that $\sup_{i\in\mathcal{I}}\|x_i\|_i<\infty$ endowed with the natural norm. The (canonical) projection map between $\prod_{i\in\mathcal{I}}X_i$ and $X_i$ given by $(x_i)_{i\in\mathcal{I}}\mapsto x_i$ will be denoted by $P_i$. In cases where $\mathcal{X}$ is finite  we use $X_1\oplus_\infty \dots\oplus_\infty X_n$ in place of the more opaque notation  $(\bigoplus_{i\leq n}X_i)_\infty$.

\begin{lemma}
\label{lemaux1}
Suppose that $\{X_i: i\in\mathcal{I}\}$ is an  indexed family of  Banach spaces. Suppose also that for each $i\in\mathcal{I}$ we are given  an $L$-orthogonal sequence $(x_{i,n})_n$ in the Banach space $X_i$.  Then   $(x_n)_n$ defined by $$x_n:=(x_{i,n})_{i \in \mathcal{I}}$$ yields  an $L$-orthogonal sequence in  the $\ell_\infty$-sum space $\left(
\bigoplus_{i \in \mathcal{I}} X_i \right)_\infty.$
\end{lemma}
\begin{proof}
Notice that for any $x=(x_i)_{i\in\mathcal{I}} \in \left(
\bigoplus_{i \in \mathcal{I}} X_i \right)_\infty$ we have
$$ \lim_n\|x+x_n\|\geq \lim_n \|x_i+x_{i,n}\|_i=1+\|x_i\|_i$$
for every $i\in \mathcal{I}$, so 
$$ 1+\|x\|\geq \lim_n\|x+x_n\| \geq \sup_{i\in \mathcal{I}}(1+\|x_i\|_i) =1+\|x\|$$
as desired.
\end{proof}
 The next technical lemma connects $L$-orthogonal elements of subspaces $X$ of $(X_1 \oplus_\infty X_2)$ with $L$-orthogonal elements of the \emph{fiber spaces} $X\cap X_i$. Given Banach spaces $X_1, X_2$ notice that $(X_1 \oplus_\infty X_2)^{**}=X_1^{**} \oplus_\infty X_2^{**}$.

\begin{lemma}
\label{lemaux2}
Suppose $X$ is a subspace of $X_1 \oplus_\infty X_2$ carrying an $L$-orthogonal element $x^{**}=x_1^{**}+x_2^{**} \in X^{**}\subseteq X_1^{**} \oplus_\infty X_2^{**}$. For each $i\in\{1,2\}$, $x_i^{**}$ is $L$-orthogonal  to $X \cap X_i$ provided $X \cap X_i \neq \{0\}$.\footnote{An element $x^{**}\in X^{**}$ is \textit{$L$-orthogonal to a subspace} $Y\subseteq X$ if $\|y+x^{**}\|=\|y\|+1$ for every $y\in Y$.} 
\end{lemma}

\begin{proof}
By symmetry, it is enough to check that $x_1^{**}$ is $L$-orthogonal to $X\cap X_1$. 
Pick any nonzero vector $x_1 \in X\cap X_1$. Then $$1<1+\|x_1\|_1=1+\|x_1\|=\|x_1+x^{**}\|=\max\{\|x_1+x_1^{**}\|_1,\|x_2^{**}\|_2\}.$$
Since $1=\|x^{**}\|\geq \|x_2^{**}\|_2$, we have $1+\|x_1\|_1=\|x_1+x_1^{**}\|_1$
as desired.
\end{proof}

\begin{cor}
\label{coroaux}
Let $X$ be a subspace of $\left(
\bigoplus_{i \in \mathcal{I}} X_i \right)_\infty$ with an $L$-orthogonal element $x^{**}\in X^{**}$.
For each $i\in\mathcal{I}$, $x_i^{**}:=P_i^{**}x^{**}$ is $L$-orthogonal to $X\cap X_i$ 
provided $X \cap X_i \neq \{0\}$.
\end{cor}
\begin{proof}
Fix $j\in \mathcal{I}$ and suppose that $X \cap X_j \neq \{0\}$. To show that $x_j^{**}:=P_j^{**}x^{**}$ is $L$-orthogonal to $X\cap X_j$ write $X$ as a subspace of $Y_j \oplus_\infty X_j$ where $Y_j:=\left(
\oplus_{i \in \mathcal{I}\setminus\{j\}} X_i \right)_\infty$ and apply the previous lemma.
\end{proof}

Now we are in position to prove Theorem \ref{theo:sucesinLorto}: 

\begin{proof}
By Lemma \ref{Ready for Main teo}, for each measure $\mu \in S_{\ell_1^{**}}$ we let a norm-one operator $T_\mu \colon \ell_1 \longrightarrow C(\{0,1\}^\nat)$ and $c_\mu \in \mathbb{R}$ such that $(T_{\mu}e_n)_n$ is an $L$-orthogonal sequence in $C(\{0,1\}^\nat)$, yet $\|T_{\mu}^{**}(\mu)-c_\mu \mathbbm{1}\|< 1$. 

\smallskip

Conveniently we stipulate $\cS:=S_{\ell_1^{**}}$ and consider $X:=\ell_1 \oplus c_0(\cS)$ endowed with the norm given by $$ \|(x,y)\|= \sup_{\mu \in \cS} \{ \|x\|_1, \|T_\mu x+y_\mu \mathbbm {1}\|_\infty\}, $$ 
being $x \in \ell_1$ and $y= (y_\mu)_{\mu \in \cS} \in c_0(\cS)$.

\begin{claim}
   $(X,\|\cdot\|)$ is a normed space isomorphic to $\ell_1\oplus_1 c_0(\cS)$.

\end{claim}
\begin{proof}[Proof of claim]
On the one hand we have	
	$$ \|(x,y)\|= \sup_{\mu  \in \cS} \{ \|x\|_1, \|T_{\mu} x+y_\mu \mathbbm {1}\|_\infty \} \leq
\sup_{\mu  \in \cS} \{ \|x\|_1, \|T_{\mu} x \|_\infty+|y_\mu| \} \leq 	$$
$$ \leq \sup_{\mu  \in \cS} \{ \|x\|_1, \|x\|_1+|y_\mu| \} = \|x\|_1 + \sup_{\mu  \in \cS} |y_\mu| = \|x\|_1 + \|y\|_\infty .$$

On the other hand 		
	$$ \|(x,y)\|= \sup_{\mu  \in \cS} \{ \|x\|_1, \|T_{\mu} x+y_\mu \mathbbm {1}\|_\infty \} \geq \sup_{\mu  \in \cS} \frac{\|x\|_1 + \|T_{\mu} x+y_\mu \mathbbm1\|_\infty }{2} \geq $$
	$$ \geq \sup_{\mu  \in \cS} \frac{\|x\|_1 + \|y_\mu \mathbbm {1}\|_\infty - \|T_{\mu} x\|_\infty }{2} \geq \sup_{\mu  \in \cS} \frac{\|x\|_1 + \|y_\mu \mathbbm {1}\|_\infty - \| x\|_1 }{2}$$
 and this latter expression equals  $\sup_{\mu  \in \cS} \frac{|y_\mu|}{2} = \frac{\| y \|_\infty }{2}$. Therefore,
	 $$  \|(x,y)\| \geq \max \left\{ \|x\|_1 , \frac{\| y \|_\infty }{2} \right\} \geq \frac{\|x\|_1 + \|y\|_\infty }{4}.$$
  Putting together the above computations we deduce the claim.
\end{proof}
The above shows that $X$  is a direct sum of $\ell_1$ and $c_0(\mathcal{S}).$

\begin{claim}
    $X$ contains an $L$-orthogonal sequence.
\end{claim}
\begin{proof}[Proof of claim]
  The operator $$\textstyle S\colon X \longrightarrow \ell_1 \oplus_\infty \left(
\bigoplus_{\mu  \in \cS} X_\mu \right)_\infty$$ given by the formula $S(x,y):=(x,(T_{\mu} x+y_\mu \mathbbm {1})_{\mu\in\mathcal{S}})$ defines an isometric embedding by the very definition of the norm in $X$, being $X_\mu=C(\{0,1\}^\nat)$ for each $\mu\in \cS$.
It follows from Lemma \ref{lemaux1} and the properties of $T_\mu$ that the sequence $(S(e_n,0))_n$ is an $L$-orthogonal sequence in $SX$, and therefore $X$ contains an $L$-orthogonal sequence.
\end{proof}

	The end game is given by the next claim:
 \begin{claim}
  $X$ does not contain an $L$-orthogonal element.
 \end{claim}
 \begin{proof}[Proof of claim]
     Suppose by way of contradiction that $z^{**} \in X^{**}$ is an $L$-orthogonal element.  Notice that the norm of $\ell_1$ and $c_0(\cS)$ as subspaces of $X$ are the standard ones; namely, $\|\cdot\|_1$ and $\|\cdot\|_\infty$, respectively. Moreover, since $X$ is a direct sum of $\ell_1$ and $c_0(\cS)$ we have that 
$$X^{**}=\ell_1^{**} \oplus c_0(\cS)^{**}= \ell_1^{**} \oplus \ell_\infty(\cS).$$
Thus, we can write $z^{**}=(x^{**},y^{**})$, where $x^{**} \in \ell_1^{**}$ and $y^{**} \in c_0(\cS)^{**}$.

\underline{Case $x^{**}=0$:} If $x^{**}=0$ then we claim that $y^{**}$ would be an $L$-orthogonal element in $c_0(\cS)^{**}$: Indeed,  the canonical isometric embedding $$i:c_0(\cS) \longrightarrow X$$ satisfies $i^{**}(y^{**})=z^{**}$ and therefore $$\|y+y^{**}\|=\|i^{**}(y+y^{**})\|=\|i(y)+z^{**}\|=\|y\|+1$$ for every $y\in c_0(\cS)$. Note that this is impossible since $c_0(\cS)$ cannot contain $L$-orthogonal elements due to the fact that it does not contain copies of $\ell_1$ (recall \cite[Theorem II.4]{god}). 

\smallskip

\underline{Case $x^{**}\neq 0$:} Suppose that $z^{**}=(x^{**},y^{**})$ with $x^{**} \neq 0$. Let $\mu:=x^{**}/\|x^{**}\| \in \mathcal{S}$.
By Corollary \ref{coroaux}, if $z^{**}$ is an $L$-orthogonal element then $P_\mu^{**}(S^{**}z^{**})$ should be an $L$-orthogonal element in $(SX) \cap X_\mu$,
where $$\textstyle P_\mu : \ell_1 \oplus_\infty \left(
\bigoplus_{\mu  \in \cS} X_\mu \right)_\infty \longrightarrow X_\mu$$ is the projection onto the coordinate $\mu$.
Notice that $$\text{$P_\mu S(x,y)=T_{\mu} x+y_\mu \mathbbm {1}$ for every $(x,y) \in X$,}$$ so
$P_\mu^{**} S^{**}(x^{**},y^{**})=T_{\mu}^{**} x^{**}+y_\mu^{**} \mathbbm {1}$. 
Since $$T_{\mu}^{**} x^{**} = T_{\mu}^{**} (\|x^{**}\|\mu) =\|x^{**}\| T_{\mu}^{**} \mu,$$
we have that $$ \|T_{\mu}^{**} x^{**}+y_\mu^{**} \mathbbm {1} -(c_\mu \|x^{**}\|+ y_\mu^{**})\mathbbm{1}\|=\|\|x^{**}\|(T_{\mu}^{**} \mu-c_\mu \mathbbm{1})\|<\|x^{**}\| \leq 1.$$

Since $SX \cap X_\mu$ contains all the constant functions, it follows from the last inequality that $P_\mu^{**}(S^{**}z^{**})$ is not $L$-orthogonal in $(SX) \cap X_\mu$, and thus $z^{**}$ cannot be an $L$-orthogonal element in $X^{**}$.  
 \end{proof}
Combining all the previous claims Theorem~\ref{theo:sucesinLorto} follows.
\end{proof}

\section{Open questions}
\label{sec:openquestions}
After our investigations, one question still remains elusive:


\begin{question}
    Is there a model of $\mathrm{ZFC}$ with $Q^+(\omega)$-measures but without $Q$-points?
\end{question}

This seems plausible, yet hard to arrange. The basic issue is the tension between the existence of $Q^+(\omega)$-measures and $Q$-points. Namely,  in all the classical models where there are no $Q$-points there are no $Q^+(\omega)$-measures either. 
We also ask whether a quantitative version version of Theorem~\ref{Theorem strong Q-points} is possible:

\begin{question}
    Assume $\mu$ is an $\varepsilon$-strong $Q$-measure and $(x_n)_{n\in \omega}$ is an $L$-orthogonal sequence in a Banach space $X$. Does
    $$\norm{x + \mu\text{-}\lim(x_n)}\ge \norm{x}+ \varepsilon\norm{\mu \text{-}\lim(x_n)}$$ hold for $x\in X$?
\end{question}

In fact, we do not have examples that separate any of the various generalizations of $Q$-points that have arisen in our work.

\begin{question}
    Are any of the notions of $Q^+(\omega)$-measure, $\delta$-fit $Q$-measure, fit $Q$-measure, $\varepsilon$-strong $Q$-measure, and the like, equivalent to each other?
\end{question}

And finally, we wonder if any of these codes the transmission of $L$-orthogonality from a sequence to an element of the bidual.

\begin{question}
    Is there a combinatorial characterization of those measures $\mu$ such that the $\mu\text{-}\lim$ through any $L$-orthogonal sequence is an $L$-orthogonal element?
\end{question}

\subsection*{Acknowledgments}
Avilés was supported by PID2021-122126NB-C32 fun\-ded by MCIN/AEI/10.13039/501100011033, 
by ERDF A way of making Europe, and by Fundaci\'{o}n S\'{e}neca - Agencia de Ciencia y Tecnolog\'{\i}a de la
Regi\'{o}n de Murcia (21955/PI/22). 
Martínez-Cervantes  was  supported by Fundaci\'{o}n S\'{e}neca - ACyT Regi\'{o}n de Murcia (grant 21955/PI/22) and by Agencia Estatal de Investigaci\'on and EDRF/FEDER ``A way of making Europe" (MCIN/AEI/10.13039/501100011033) (grant PID2021-122126NB-C32). 
Poveda is funded by the Department of Mathematics and the Center of Mathematical Sciences and Applications at Harvard University. 
Sáenz gratefully acknowledges support received from the PAPIIT grant IN 101323, the CONAHCYT scholarship and by Fundaci\'{o}n S\'{e}neca - Agencia de Ciencia y Tecnolog\'{\i}a de la
Regi\'{o}n de Murcia (21955/PI/22).

\bibliographystyle{alpha} 
\bibliography{citations}
\end{document}